\documentclass[a4paper,12pt,oneside,headsepline]{scrartcl}

\usepackage{etoolbox}
\usepackage{ifdraft}
\usepackage{verbatim}

\usepackage[utf8]{inputenc}
\usepackage[T1]{fontenc}

\usepackage{lmodern}
\usepackage{fourier}

\usepackage{amssymb, amsfonts} 

\usepackage{mathrsfs} 
\usepackage{dsfont}
\usepackage[usenames,dvipsnames]{xcolor}
\usepackage{lettrine}
\usepackage{url}

\usepackage{stmaryrd}

\usepackage[english]{babel}

\usepackage[draft=false]{scrlayer-scrpage}

\usepackage{wrapfig}
\usepackage{footmisc}
\usepackage{needspace}

\usepackage{amsmath}
\usepackage{mathtools}
\usepackage{nccmath}
\usepackage[thmmarks, amsmath, amsthm]{ntheorem}

\providebool{nobiblatex}
\boolfalse{nobiblatex}
\ifbool{nobiblatex}{%
}{%
  \usepackage[style=numeric-comp,giveninits,url=false,sortcites=true,maxbibnames=99]{biblatex}
  \bibliography{bibliography.bib}
}

\usepackage[pdftex,final,
            pdfborder={0 0 0}, colorlinks=true,
            linkcolor=BrickRed, citecolor=ForestGreen, urlcolor=RoyalBlue]{hyperref}

\ifdraft{%
\usepackage[textsize=scriptsize,colorinlistoftodos]{todonotes}
\usepackage{lineno}
}{}

\usepackage{booktabs}
\usepackage[inline]{enumitem}
\usepackage{float}
\usepackage{graphicx}
\usepackage[final]{listings}
\usepackage{multicol}
\usepackage{scrtime}
\usepackage{tikz}
\usetikzlibrary{matrix,arrows,positioning}

\ifdraft{\date{\today}}{\date{}}

\ifdraft{%
\linenumbers
\newcommand*\patchAmsMathEnvironmentForLineno[1]{%
  \expandafter\let\csname old#1\expandafter\endcsname\csname #1\endcsname
  \expandafter\let\csname oldend#1\expandafter\endcsname\csname end#1\endcsname
  \renewenvironment{#1}%
     {\linenomath\csname old#1\endcsname}%
     {\csname oldend#1\endcsname\endlinenomath}}%
\newcommand*\patchBothAmsMathEnvironmentsForLineno[1]{%
  \patchAmsMathEnvironmentForLineno{#1}%
  \patchAmsMathEnvironmentForLineno{#1*}}%
\AtBeginDocument{%
\patchBothAmsMathEnvironmentsForLineno{equation}%
\patchBothAmsMathEnvironmentsForLineno{align}%
\patchBothAmsMathEnvironmentsForLineno{flalign}%
\patchBothAmsMathEnvironmentsForLineno{alignat}%
\patchBothAmsMathEnvironmentsForLineno{gather}%
\patchBothAmsMathEnvironmentsForLineno{multline}%
}}

\KOMAoptions{DIV=10}

\clearmainofpairofpagestyles
\automark[section]{subsection}

\renewcommand{\subsectionmark}[1]{}

\lohead{{\small \headertitle}}
\rohead{{\small \headerauthors}}

\RedeclareSectionCommand[%
  font=\Large\sffamily\bfseries,%
  beforeskip=1\baselineskip,%
  afterskip=0.5\baselineskip,%
  indent=0em%
  ]{section}

\RedeclareSectionCommands[%
  font=\normalfont\bfseries,%
  afterskip=-1em%
  ]{subsection,subsubsection}

\RedeclareSectionCommands[%
  font=\normalfont\itshape,%
  afterskip=-1em,%
  indent=0pt,
  ]{paragraph}

\ifbool{nobiblatex}{}{%
  \AtEveryBibitem{\clearfield{doi}}
  \AtEveryBibitem{\clearfield{isbn}}
  \AtEveryBibitem{\clearfield{issn}}
  \AtEveryBibitem{\clearfield{pages}}
  \AtEveryBibitem{\clearlist{language}}

  \setlength\bibitemsep{1pt}
  
  \renewbibmacro*{in:}{}
  \DeclareFieldFormat
    [article,inbook,incollection,inproceedings,patent,thesis,unpublished]
    {title}{#1}
}

\newenvironment{enumeratearabic}{
\begin{enumerate}[label=(\arabic*), leftmargin=0pt,labelindent=2em,itemindent=!]
}{
\end{enumerate}
}

\newenvironment{enumerateroman}{
\begin{enumerate}[label=(\roman*), leftmargin=0pt,labelindent=2em,itemindent=!]
}{
\end{enumerate}
}

\newenvironment{enumeratearabic*}{
\begin{enumerate*}[label=(\arabic*)] 
}{
\end{enumerate*}
}

\newenvironment{enumerateroman*}{
\begin{enumerate*}[label=(\roman*)] 
}{
\end{enumerate*}
}

\numberwithin{equation}{section}

\theoremnumbering{arabic}
\newtheorem{theoremcounter}{theoremcounter}[section]
\theoremnumbering{Alph}
\newtheorem{maintheoremcounter}{maintheoremcounter}

\theoremstyle{plain}

\newtheorem{corollary}[theoremcounter]{Corollary}
\newtheorem{lemma}[theoremcounter]{Lemma}

\newtheorem{proposition}[theoremcounter]{Proposition}
\newtheorem{theorem}[theoremcounter]{Theorem}

\theoremstyle{plain}

\newtheorem{maintheorem}[maintheoremcounter]{Theorem}
\newtheorem{mainclassification}[maintheoremcounter]{Classification}

\theoremstyle{definition}

\newtheorem{definition}[theoremcounter]{Definition}
\newtheorem{example}[theoremcounter]{Example}

\newtheorem{mainexample}[maintheoremcounter]{Example}

\theoremstyle{remark}

\newtheorem{remark}[theoremcounter]{Remark}

\theoremstyle{nonumberremark}

\newtheorem{mainremark}{Remark}

\newenvironment{mainremarkenumerate}
{%
\mainremark
\enumeratearabic
}{%
\endenumeratearabic
\endmainremark
}%

\lstdefinelanguage{julia}
{%
  keywordsprefix=\@,
  morekeywords={%
    error,throw,assert,
    begin,end,
    if,elseif,else,
    try,catch,
    while,for,do,break,continue,
    function,return,where,
    type,mutable,abstract,const,
    in,
    macro,
    quote,
    let,
    ccall,
    using,module,import,export,importall,
    local,global,
    Any,Nothing,None,Nullable,
  },
  sensitive=true,
  morecomment=[l]{\#},
  morestring=[b]",
}

\newcommand{\tx}{\ensuremath{\text}}
\newcommand{\thdash}{\nbd th}
\newcommand{\nbd}{\nobreakdash-\hspace{0pt}}

\ifdraft{
 \newcommand{\texpdf}[2]{#1}
}{
 \newcommand{\texpdf}[2]{\texorpdfstring{#1}{#2}}
}

\newcommand{\minwidthmathbox}[2]{%
  \mathmakebox[{\ifdim#1<\width\width\else#1\fi}]{#2}%
}

\newcommand{\tbf}{\bfseries}

\newcommand{\cF}{\ensuremath{\mathcal{F}}}

\newcommand{\frake}{\ensuremath{\mathfrak{e}}}

\newcommand{\frakk}{\ensuremath{\mathfrak{k}}}

\newcommand{\frakz}{\ensuremath{\mathfrak{z}}}

\newcommand{\frakA}{\ensuremath{\mathfrak{A}}}
\newcommand{\frakB}{\ensuremath{\mathfrak{B}}}

\newcommand{\rmd}{\ensuremath{\mathrm{d}}}

\newcommand{\rmF}{\ensuremath{\mathrm{F}}}

\newcommand{\rmH}{\ensuremath{\mathrm{H}}}

\newcommand{\rmL}{\ensuremath{\mathrm{L}}}

\newcommand{\rmR}{\ensuremath{\mathrm{R}}}

\newcommand{\wtd}{\widetilde}
\newcommand{\ov}{\overline}

\NewCommandCopy{\rightarroworig}{\rightarrow}
\renewcommand{\rightarrow}
  {\DOTSB\protect\relbar\mspace{-9.7mu}\rightarroworig}
\renewcommand{\hookrightarrow}
  {\DOTSB\protect\lhook\joinrel\mspace{-0.1mu}\relbar\mspace{-11.9mu}\rightarroworig}
\renewcommand{\twoheadrightarrow}
  {\DOTSB\protect\rightarroworig\mspace{-15mu}\rightarroworig}

\NewCommandCopy{\leftarroworig}{\leftarrow}
\renewcommand{\leftarrow}
  {\DOTSB\protect\leftarroworig\mspace{-9.7mu}\relbar}

\renewcommand{\longrightarrow}
  {\DOTSB\protect\relbar\joinrel\mspace{-0.2mu}\relbar\mspace{-9.5mu}\rightarroworig}

\newcommand{\xtwoheadrightarrow}[1]
  {\DOTSB\protect\xrightarrow{#1}\mspace{-15mu}\rightarroworig}

\renewcommand{\xlongrightarrow}[1]
  {\xrightarrow{\minwidthmathbox{1.256em}{#1}}}

\newcommand{\xlongmapsto}[1]{\DOTSB\mapstochar\xlongrightarrow{#1}}

\newcommand{\ra}{\ensuremath{\rightarrow}}
\newcommand{\hra}{\ensuremath{\hookrightarrow}}
\newcommand{\thra}{\ensuremath{\twoheadrightarrow}}

\newcommand{\lra}{\ensuremath{\longrightarrow}}

\newcommand{\mto}{\ensuremath{\mapsto}}
\newcommand{\lmto}{\ensuremath{\longmapsto}}

\renewcommand{\Re}{\ensuremath{\mathrm{Re}}}
\renewcommand{\Im}{\ensuremath{\mathrm{Im}}}

\renewcommand{\pmod}[1]{\ensuremath{\;(\mathrm{mod}\, #1)}}

\newcommand{\sgn}{\ensuremath{\mathrm{sgn}}}

\newcommand{\Hom}{\ensuremath{\mathrm{Hom}}}

\newenvironment{psmatrix}{\left(\begin{smallmatrix}}{\end{smallmatrix}\right)}
\newcommand{\Mat}[1]{\ensuremath{\mathrm{Mat}_{#1}}}

\newcommand{\lspan}{\ensuremath{\mathop{\mathrm{span}}}}

\newcommand{\ZZ}{\ensuremath{\mathbb{Z}}}

\newcommand{\RR}{\ensuremath{\mathbb{R}}}
\newcommand{\CC}{\ensuremath{\mathbb{C}}}

\newcommand{\GL}[1]{\ensuremath{\mathrm{GL}_{#1}}}

\newcommand{\SL}[1]{\ensuremath{\mathrm{SL}_{#1}}}
\newcommand{\SO}[1]{\ensuremath{\mathrm{SO}_{#1}}}

\newcommand{\HS}{\mathbb{H}}

\newcommand{\Ga}{\Gamma}
\newcommand{\ga}{\gamma}

\newcommand{\cok}{\mathrm{cok}}

\newcommand{\Ad}{\mathrm{Ad}}
\newcommand{\Lie}{\mathrm{Lie}}

\newcommand{\Xm}{X_-}
\newcommand{\Xz}{X_\ast}
\newcommand{\Xp}{X_+}
\newcommand{\Ym}{Y_-}
\newcommand{\Yz}{Y_\ast}
\newcommand{\Yp}{Y_+}

\newcommand{\Pol}{\mathrm{Pol}}
\newcommand{\End}{\mathrm{End}}

\newcommand{\tdtr}{\wtd{\mathrm{tr}}}

\newcommand{\img}{\mathrm{img}}

\DeclareMathOperator{\Rep}{\mathsf{Rep}}
\DeclareMathOperator{\HC}{\mathsf{HC}}
\newcommand{\lieg}{\mathfrak{g}}
\newcommand{\idm}{\mathfrak{m}}

\newcommand{\dD}{\mathfrak{D}}
\newcommand{\dA}{\mathfrak{A}}
\newcommand{\dB}{\mathfrak{B}}
\newcommand{\kA}{\mathcal A}

\newcommand{\EE}{\mathbb{E}}
\newcommand{\HH}{\mathbb{H}}

\newcommand{\llbrace}{(\!(}
\newcommand{\rrbrace}{)\!)}

\newcommand{\lar}{\longrightarrow}

\newcommand{\headertitle}{{\normalfont%
  A classification of polyharmonic Maa\ss{} forms
}}
\newcommand{\headerauthors}{%
  C.~Alfes-Neumann,
  I.~Burban,
  M.~Raum%
}
\title{%
  A classification of\\polyharmonic Maa\ss{} forms\\via quiver representations
}
\author{%
Claudia Alfes-Neumann%
\thanks{The author was partially supported by the Daimler and Benz Foundation, the Klaus Tschira Boost Fund and the German
Research Foundation SFB-TRR 358/1 2023 -- 491392403.}
\and%
Igor Burban%
\thanks{The author was partially supported by the German
Research Foundation SFB-TRR 358/1 2023 -- 491392403.}
\and
Martin Raum%
\thanks{The author was partially supported by Vetenskapsr\aa det Grant~2019-03551.}%
}
\ifdraft{%
\date{\today\ at\ \thistime}
}

\begin{document}

\thispagestyle{scrplain}
\begingroup
\deffootnote[1em]{1.5em}{1em}{\thefootnotemark}
\maketitle
\endgroup


\begin{abstract}
\small
\noindent
{\tbf Abstract:}
We give a classification of the Harish-Chandra modules generated by the pullback to~$\SL{2}(\RR)$ of \emph{poly}harmonic Maa\ss{} forms for congruence subgroups of~$\SL{2}(\ZZ)$ with exponential growth allowed at the cusps. This extends results of Bringmann--Kudla in the harmonic case. While in the harmonic setting there are nine cases, our classification comprises ten; A new case arises in weights $k > 1$. 
To obtain the classification we introduce quiver representations into the topic and show that those associated with polyharmonic Maa\ss{} forms are cyclic, indecomposable representations of the two-cyclic or the Gelfand quiver. 
 A classification of these transfers to a classification of polyharmonic weak Maa\ss{} forms. To realize all possible cases of Harish-Chandra modules we develop a theory of weight shifts for Taylor coefficients of vector-valued spectral families.
We provide a comprehensive computer implementation of this theory, which allows us to provide explicit examples.
\\[.3\baselineskip]
\noindent
\textsf{\textbf{%
  polyharmonic Maa\ss{} forms%
}}%
\noindent
\ {\tiny$\blacksquare$}\ %
\textsf{\textbf{%
  Harish-Chandra modules%
}}%
\noindent
\ {\tiny$\blacksquare$}\ %
\textsf{\textbf{%
  Gelfand quiver%
}}%
\noindent
\ {\tiny$\blacksquare$}\ %
\textsf{\textbf{%
  mock modular forms%
}}%
\noindent
\ {\tiny$\blacksquare$}\ %
\textsf{\textbf{%
  Kronecker limit formula%
}}%
\\[.2\baselineskip]
\noindent
\textsf{\textbf{%
  MSC Primary:
  11F37, 17B10%
}}%
\ {\tiny$\blacksquare$}\ %
\textsf{\textbf{%
  MSC Secondary:
  11F72, 16G20%
}}
\end{abstract}

\Needspace*{4em}
\addcontentsline{toc}{section}{Introduction}
\markright{Introduction}
\lettrine[lines=2,nindent=.2em]{\tbf H}{armonic} weak Maa\ss{} forms have proved to be immensely useful in number theory and related areas such as arithmetic geometry, combinatorics, and mathematical physics. Numerous applications were showcased by Bringmann--Fol\-som--Ono--Rolen~\cite{bringmann-folsom-ono-rolen-2018}, often rooted in generating series that are ``completed'' to produce harmonic weak Maa\ss{} forms. Special highlights are the confirmation of the Andrews--Dra\-go\-nette Conjecture~\cite{bringmann-ono-2006}, incoherent Eisenstein series in the Kudla Program~\cite{kudla-rapoport-yang-1999}, and Mathieu Moonshine~\cite{eguchi-ooguri-tachikawa-2011,gannon-2016} associated with~$K3$--surfaces. 
As work by Duke, \.Imamo\u{g}lu, and T\'oth~\cite{dit2013} shows \emph{poly}harmonic weak Maa\ss{} forms play an evenly important role in number theory. Moreover, these forms appear in string theory, especially via the theory of modular graph functions~\cite{green-vanhove-2000}, and in the correspond mathematical theory of iterated period integrals on the moduli space of marked genus-$1$ curves~\cite{brown-2017-preprint,brown-2018}.

The study of harmonic weak Maa\ss{} forms from a representation theoretic perspective was initiated by Bringmann and Kudla~\cite{BK}. They provided a classification of the Harish-Chandra modules generated by the pullback to $\mathrm{SL}_2(\RR)$ of harmonic weak Maa\ss{} forms of integral weight. 

In the present work we study polyharmonic (weak) Maa\ss{} forms, i.e.\@ forms that vanish under a power of the Laplace operator, transform like modular forms, and have at most exponential growth at the cusps. Here, the weight $k$ Laplace operator is defined as
\begin{gather*}
\Delta_k = - \rmR_{k-2}\rmL_k, \quad \tx{where\ } \rmR_{k} = 2 i \partial_\tau +ky^{-1}, \, \rmL_k=-2iy^2\partial_{\ov{\tau}}
\tx{.}
\end{gather*}
We call a polyharmonic Maa\ss{} form~$f$ of weight~$k$ satisfying $\Delta^{d+1}_k f=0$ and $\Delta^d_k f\neq 0$ polyharmonic of~\emph{exact depth~$d$} (compare Definition \ref{def:polyharmonic_maass_forms}).
 A classical example of such a form, which appears already in the Kronecker limit formula, is the function
\begin{gather*}
  - \tfrac{1}{6}\, \mathrm{log}\big( y^{12} \big| \eta(\tau) \big|^{48} \big)
\tx{,}
\end{gather*}
where $\eta(\tau) = \exp(\pi i \tau \slash 12)\, \prod_{n=1}^\infty (1-\exp(2\pi i n \tau))$.

\subsection*{The classification}

Our first main result is the classification of Harish-Chandra modules corresponding to polyharmonic Maa\ss{} forms of integral weight. As a special case we recover the classification for harmonic Maa\ss{} forms obtained by Bringmann and Kudla. Harish-Chandra modules associated with polyharmonic Maa\ss{} forms are formally described in Section~\ref{ssec:automorphic_forms}. Informally, they capture all~$\SL{2}(\RR)$-covariant differential equations satisfied by a given form.

\begin{mainclassification}%
\label{mainthm:classification}
In Sections~\ref{ssec:classification_representation_labels} and~\ref{ssec:classification_BK_labels} we give a classification comprising\/~$10$ cases of Harish-Chandra modules associated with weight-$k$ polyharmonic weak Maa\ss{} forms of exact depth~$d$: Cases~Ia--d for~$k < 1$, Cases~IIa--b for~$k = 1$, and Cases~IIIa--d for~$k > 1$. The classification is given in terms of the vanishing and non-vanishing of the following functions:
\begin{center}
\begin{tabular}{@{}lclcl@{}}
  $k < 1$ && $k = 1$ && $k > 1$
\\
\cmidrule{1-1}
\cmidrule{3-3}
\cmidrule{5-5}
   $\rmL_k \Delta_k^d\, f \stackrel{?}{=} 0$, $\rmR_k^{1-k} \Delta_k^d\, f \stackrel{?}{=} 0$
&& $\rmL_k \Delta_k^d\, f \stackrel{?}{=} 0$
&& $\rmL^k \Delta_k^{d-1}\, f \stackrel{?}{=} 0$, $\rmL_k \Delta_k^d\, f \stackrel{?}{=} 0$, $\rmL_k^k \Delta_k^d\, f \stackrel{?}{=} 0$.
\end{tabular}
\end{center}
\end{mainclassification}

\begin{mainremarkenumerate}%
\label{mainrm:classification}
\item
The three vanishing conditions that appear for~$k > 1$ imply one another from left to right, so that they yield~$4$ as opposed to~$8$ cases.
\item
The classification of harmonic weak Maa\ss{} forms, i.e.\@ the case~$d = 0$, encompasses only~$9$ cases. Each of them generalizes in a suitable way, but there is an additional Case~IIId if~$k > 1$, which occurs only in positive depth~$d$.
\end{mainremarkenumerate}

\begin{mainexample}%
\label{mainex:classification}
In our classification, the previously mentioned form~$-\frac{1}{6} \log(y^{12} |\eta(\tau)|^{48})$ falls into Case~Ia. Bringmann--Kudla highlighted~$s$-derivatives of Eisenstein series as a further class of interesting functions that are not covered by their classification. These are polyharmonic Maa\ss{} forms of type~Id if~$k < 0$, of type~Ia if~$k = 0$, of type~IIIb if~$k = 2$, and of type~IIIa if~$k > 2$. The case~$k = 1$ gives rise to incoherent Eisenstein series that fall under type~IIa. See Section \ref{sec:modular_realization} for details.
\end{mainexample}

Elementary Lie algebra considerations that serve well in the harmonic setup cease to work in ours. 
In particular, the method of classification employed by Bringmann--Kudla relies heavily on the harmonicity condition~$\Delta_k\, f=0$, which implies that $K$-types of the Harish-Chandra module associated with~$f$ occur with multiplicity at most~$1$. The transition between $K$-types can therefore be adequately described by the vanishing or non-vanishing of scalars. This no longer holds true for general polyharmonic Maa\ss{} forms, which is the major obstacle to the classification that we achieved. It turns out that the theory of quiver representations beautifully enables us to circumvent this obstacle. In Proposition~\ref{P:Equivalence} and Theorem~\ref{T:Equivalence}, we provide equivalences between categories of specific Harish-Chandra modules and quiver representations following~\cite{Gelfand,mp1,mp,CB,Nodal,Gnedin}. This naturally leads us to representations of the Gelfand quiver, which are well-known to be intricate. A key step in our classification is Theorem~\ref{thm:corrspondence_polyharmonic_maass_to_harish_chandra}, in which we show that quiver representations that arise from polyharmonic Maa\ss{} forms are cyclic and indecomposable. There are only few cyclic representations of the Gelfand quiver, which we give in Theorem~\ref{T:GelfandQuiverCyclicReps} and which mirror the Cases~Ia--d and~IIIa--d in our classification. Cases~IIa and~IIb arise from the two-cyclic quiver as discussed in Remark~\ref{rm:CyclicQuiverCyclicReps}.

\subsection*{The explicit realization}

Our second main result is the explicit realization of all cases of the classification. Classification~\ref{mainthm:classification} constrains the possible Harish-Chandra modules associated with polyharmonic weak Maa\ss{} forms, however a priori it is not clear whether each case appears.
\begin{maintheorem}%
\label{mainthm:realization}
For any~$k \in \ZZ$ and any case of Classification~\ref{mainthm:classification} associated with this~$k$, there is a (vector-valued) polyharmonic weak Maa\ss\ form of weight~$k$ that realizes this case.
\end{maintheorem}

\begin{mainremark}
In Theorem~\ref{thm:spectral_families_altering_weight} we provide a comprehensive theory that allows us to alter the weight of a modular realization of Cases~Ia--d or IIIa--d if it is obtained from a spectral family. This theory involves a detailed analysis of the action of several differential operators on spectral families and is interesting in its own right.
\end{mainremark}

To prove Theorem~\ref{mainthm:realization} we construct preimages under $\Delta_k^d$ of the realizations provided by Bringmann and Kudla. For scalar-valued realizations we can employ a classical approach and extract polyharmonic Maa\ss{} forms from the Taylor coefficients of a spectral family. This idea was already used by  Lagarias and Rhoades~\cite{lagariasrhoades} for Eisenstein series and by Duke--\.Imamo\u{g}lu--T\'oth~\cite{dit2013} for Poincar\'e series. For vector-valued forms we vastly generalize this method. We show that the weight of a modular realization can be adjusted when taking products with specific vector-valued Maa\ss{} forms. In special cases this was also used by Bringmann and Kudla.  Our Theorem~\ref{thm:spectral_families_altering_weight} extends this approach to the broadest context compatible with our classification.

The proof of Theorem~\ref{thm:spectral_families_altering_weight} depends on a delicate analysis of the action of the Laplace operator on specific products of Maa\ss{} forms. It can be viewed as an explicit instance of a theory by Bernstein--Gelfand~\cite{bernstein-gelfand-1980}, who studied tensor products of Harish-Chandra modules with finite dimensional representation. Our treatment is sufficiently explicit to accommodate a computer implementation, which is available on the third author's homepage and allows for the precise calculation of coefficients in every single case. 

\subsection*{Structure of the paper}
The paper is structured as follows. In Section~\ref{sec:harish_chandra_blocks} we provide preliminaries on Harish-Chandra modules and investigate the structure of the relevant path algebra. In Section~\ref{sec:cyclic_modules} we classify the corresponding cyclic modules. The bridge between these results and the classification of the Harish-Chandra modules arising from polyharmonic weak Maa\ss{} forms is provided in Section~\ref{sec:polyharmonic_to_quiver}. In Section~\ref{sec:spectral_taylor_coefficients} we develop the results that enable us to alter the weight of spectral families. We use these results to show that there exist polyharmonic weak Maa\ss{} forms for each of the cases of the classification in Section~\ref{sec:modular_realization}.




\section{Blocks of the category of Harish-Chandra--modules for \texpdf{$\SL{2}(\RR)$}{SL(2,R)}}%
\label{sec:harish_chandra_blocks}

\subsection{Preliminaries on Harish-Chandra--modules}

Let $G = \SL{2}(\RR)$ and $K = \SO{2}(\RR)$ be a maximal compact subgroup. Let
\begin{gather*}
  \lieg
:=
  \Lie(G) \otimes_{\RR} \CC
\cong
  \mathfrak{sl}_2(\CC)
=
  \left\langle
  H = 
  \begin{psmatrix}
  0 & -i \\
  i & 0
  \end{psmatrix}
\tx{,}\,
  X = \mfrac{1}{2}
  \begin{psmatrix}
  1 & i \\
  i & -1
  \end{psmatrix}
\tx{,}\,
  Y = 
  \mfrac{1}{2}\begin{psmatrix}
  1 & -i \\
  -i & -1
  \end{psmatrix}
  \right\rangle
\end{gather*}
be the complexified Lie algebra of~$G$ and~$U(\lieg)$ be the universal enveloping algebra of~$\lieg$. The following relations in $\lieg$ are satisfied:
\begin{gather}\label{E:DefiningIdentities}
  [H, X] = 2 X
\tx{,\ }
  [H, Y] = - 2Y
\tx{\ and\ }
  [X, Y] = H
\tx{.}
\end{gather}
Let 
\begin{gather}
\label{eq:def:casimir_element}
  C = H^2-2H + 4 XY = H^2 + 2H + 4 YX \in U(\lieg) 
\end{gather}
be a Casimir element. It is well-known that the center~$\frakz$ of~$U(\lieg)$ is~$\CC[C]$ (the algebra of polynomials in $C$). For any~$\theta \in \RR$, we set
\begin{gather}
\label{E:Rotation}
  k_\theta = \exp(i \theta H)
= 
  \begin{psmatrix}
  \cos(\theta) & \sin(\theta) \\
  -\sin(\theta) & \cos(\theta)
  \end{psmatrix}
\in
  K = \SO{2}(\RR)
\tx{.}
\end{gather}

\begin{definition}
A complex vector space~$M$ is a~$(\lieg, K)$--module if it has 
a structure $(M, \circ)$ of a representation of~$\lieg$ as well as a structure~$(M, \cdot)$ of a locally smooth representation of~$K$ such that, first,
\begin{gather*}
  \bigl(\Ad_k(Z)\bigr) \circ v
=
  k \cdot \bigl(Z \circ (k^{-1} \cdot v)\bigr)
\end{gather*}
for any~$v \in M$, $k \in K$, and~$Z \in \lieg$, where~$\Ad$ is the adjoint action of~$G$ on~$\lieg$, and second,
\begin{gather*}
  \mfrac{\rmd}{\rmd t}\, \big( \exp(tZ) \cdot v \big) \big|_{t = 0}
=
  Z \circ v
\end{gather*}
for any~$v \in M$ and~$Z \in \frakk \subset \lieg$, the Lie complexified Lie algebra of~$K$. In what follows, we shall omit~$\circ$ and~$\cdot$ from the notation for the action on~$M$. For~$n \in \ZZ$ we put:
\begin{gather*}
  M_n = \big\{v \in M \,:\, H \circ v = n v \big\}
\tx{.}
\end{gather*}
A $(\lieg, K)$--module $M$ is a \emph{Harish-Chandra} module if it is finitely generated over~$U(\lieg)$ and 
\begin{gather*}
  M
\cong
  \bigoplus_{n \in \ZZ} M_n
\tx{\ with\ }
  \dim_{\CC}(M_n) < \infty
  \tx{\ for all\ } n \in \ZZ
\tx{.}
\end{gather*}
\end{definition}

In what follows, $\HC(\lieg, K)$ denotes the category of Harish-Chandra~$(\lieg, K)$--modules. 

\begin{definition}
Let~$l \in \ZZ_{\ge 0}$ and~$\gamma = l^2-1$. Let~$\HC_l(\lieg, K)$ be the full subcategory of $\HC(\lieg, K)$ consisting of such Harish-Chandra modules~$M$ for which there exists~$m \in \ZZ_{> 0}$ (depending on $M$) such that~$(C- \gamma I)^m \cdot M = 0$.
\end{definition}

We have the following standard result.
\begin{lemma}
\label{la:hc_quiver_category}
The category\/~$\HC_l(\lieg, K)$ is a full subcategory of\/~$\HC(\lieg, K)$. For any object~$M$ of\/~$\HC_l(\lieg, K)$ we have a direct sum decomposition 
\begin{gather*}
  M = \bigoplus\limits_{i \in \ZZ} M_{-l-1 + 2i}
\end{gather*}
and its $\lieg$--module structure can be visualized by the following (infinite) diagram of finite dimensional vector spaces and linear maps:
\begin{gather}
\label{E:DiagramHCmodule}
  \begin{tikzpicture}[baseline]
  \matrix(m)[matrix of math nodes,
  column sep = 2.5em,
  ampersand replacement=\&]
  {  \cdots \& M_{-l-1} \& M_{-l+1}
  \& \cdots \& M_{l-1}  \& M_{l+1}
  \& \cdots\tx{.} \\};
  \path[-stealth,bend left=15]
  (m-1-1.east |- m-1-2.north) edge node[above] {$X$} (m-1-2.west |- m-1-2.north)
  (m-1-2.east |- m-1-2.north) edge node[above] {$X$} (m-1-3.west |- m-1-2.north)
  (m-1-3.east |- m-1-2.north) edge node[above] {$X$} (m-1-4.west |- m-1-2.north)
  (m-1-4.east |- m-1-2.north) edge node[above] {$X$} (m-1-5.west |- m-1-2.north)
  (m-1-5.east |- m-1-2.north) edge node[above] {$X$} (m-1-6.west |- m-1-2.north)
  (m-1-6.east |- m-1-2.north) edge node[above] {$X$} (m-1-7.west |- m-1-2.north)
  (m-1-7.west |- m-1-2.south) edge node[below] {$Y$} (m-1-6.east |- m-1-2.south)
  (m-1-6.west |- m-1-2.south) edge node[below] {$Y$} (m-1-5.east |- m-1-2.south)
  (m-1-5.west |- m-1-2.south) edge node[below] {$Y$} (m-1-4.east |- m-1-2.south)
  (m-1-4.west |- m-1-2.south) edge node[below] {$Y$} (m-1-3.east |- m-1-2.south)
  (m-1-3.west |- m-1-2.south) edge node[below] {$Y$} (m-1-2.east |- m-1-2.south)
  (m-1-2.west |- m-1-2.south) edge node[below] {$Y$} (m-1-1.east |- m-1-2.south);
  \end{tikzpicture}
\end{gather}
\end{lemma}

We call~\eqref{E:DiagramHCmodule} a \emph{diagram description} of a Harish-Chandra module~$M \in \HC_l(\lieg, K)$ (see also~\cite{Mazorchuk2010} for 
a detailed treatment of various classes of representations of $\lieg$).

For any~$p \in \ZZ$, consider the following fragment of~\eqref{E:DiagramHCmodule}:
\begin{gather*}
  \begin{tikzpicture}[baseline]
  \matrix(m)[matrix of math nodes,
  column sep = 4em,
  ampersand replacement=\&]
  { M_{p-1} \& M_{p+1}. \\};
  \path[-stealth,bend left=15]
  (m-1-1.north east) edge node[above] {$X$} (m-1-2.north west)
  (m-1-2.south west) edge node[below] {$Y$} (m-1-1.south east);
  \end{tikzpicture}
\end{gather*}
A calculation shows that
\begin{gather}
\label{E:KeyIdentities}
\begin{aligned}
  & XY = \mfrac{1}{4}\bigl(C - H^2 + 2H\bigr) = \mfrac{1}{4}\bigl(C -(p^2-1)\bigr)
\tx{,}\\ 
  & YX = \mfrac{1}{4}\bigl(C - H^2 - 2H\bigr) = \mfrac{1}{4}\bigl(C -(p^2-1)\bigr) 
\tx{.}
\end{aligned}
\end{gather}
As a consequence of~\eqref{E:KeyIdentities}, we obtain the following statements about the diagram description~\eqref{E:DiagramHCmodule} of~$M \in \HC_l(\lieg, K)$:
\begin{itemize}
\setlength\itemsep{0pt}
\item $X|_{M_{p-1}}$ is an isomorphism for $p \ne \pm l$. 
\item $Y|_{M_{p+1}}$ is an isomorphism for $p \ne \pm l$.
\item For $p = \pm l$, the corresponding endomorphisms $XY$ and $YX$ are nilpotent.
\end{itemize}
Let $M \xlongrightarrow{\psi} N$ be a morphism in the category $\HC_l(\lieg, K)$. Then for any $i \in \ZZ$ we have: $\psi(M_i) \subseteq N_i$ and for any $p \in \ZZ$ the following diagram is commutative:
\begin{gather}
\label{E:MorphismOfReprs}
  \begin{tikzpicture}[baseline]
  \matrix(m)[matrix of math nodes,
  column sep = 4em, row sep = 4em,
  ampersand replacement=\&]
  {  M_{p-1} \& M_{p+1} 
  \\ N_{p-1} \& N_{p+1} 
  \\};
  \path[-stealth]
  (m-1-1) edge node[left]  {\small $\psi_{p-1}$} (m-2-1)
  (m-1-2) edge node[right] {\small $\psi_{p+1}$} (m-2-2);
  \path[-stealth,bend left=10]
  (m-1-1.east |- m-1-1.north) edge node[above] {\small $X_M$} (m-1-2.west |- m-1-1.north)
  (m-1-2.west |- m-1-1.south) edge node[below] {\small $Y_M$} (m-1-1.east |- m-1-1.south);
  \path[-stealth,bend left=10]
  (m-2-1.east |- m-2-1.north) edge node[above] {\small $X_N$} (m-2-2.west |- m-2-1.north)
  (m-2-2.west |- m-2-1.south) edge node[below] {\small $Y_N$} (m-2-1.east |- m-2-1.south);
  \end{tikzpicture}
\end{gather}
where $\psi_{p\pm 1} := \psi\big|_{M_{p \pm 1}}$. Conversely, let $M, N \in \HC_l(\lieg, K)$ and $\bigl(M_i \xlongrightarrow{\psi_i} N_i\bigr)_{i \in \ZZ}$ be any family of linear maps such that the diagram \eqref{E:MorphismOfReprs} is commutative for any $p \in \ZZ$.  We put $\psi = \oplus_{i \in \ZZ} \psi_i$. Then $M \xlongrightarrow{\psi} N$ is a morphism in the category $\HC_l(\lieg, K)$.

\subsection{Path algebras}

Let $\dD = \CC\llbracket t\rrbracket$, $\idm = (t)$ and
\begin{gather}
\label{E:GelfandOrder}
  \dA = 
  \begin{psmatrix}
  \dD & \idm & \idm\\
  \dD & \dD & \idm \\
  \dD & \idm & \dD
  \end{psmatrix}
\subset
  \Mat{3\times 3}(\dD)
\quad\tx{and}\quad 
  \dB =
  \begin{psmatrix}
  \dD & \idm \\
  \dD & \dD 
  \end{psmatrix}
\subset
  \Mat{2\times 2}(\dD)
\tx{.}
\end{gather}
Note that the algebra~$\dB$ is isomorphic to the arrow ideal completion of the path algebra of the cyclic quiver
\begin{gather}
\label{E:CyclicQuiver}
  \begin{tikzpicture}[baseline]
  \matrix(m)[matrix of math nodes,
  column sep = 4em,
  ampersand replacement=\&]
  { - \& + \\};
  \path[-stealth,bend left=15]
  (m-1-1.east |- m-1-1.north) edge node[above] {$a$} (m-1-2.west |- m-1-1.north)
  (m-1-2.west |- m-1-1.south) edge node[below] {$b$} (m-1-1.east |- m-1-1.south);
  \end{tikzpicture}
\tx{,}
\end{gather}
whereas~$\dA$ is isomorphic to the arrow ideal completion of the path algebra of the so-called Gelfand quiver
\begin{gather}
\label{E:GelfandQuiver}
  \begin{tikzpicture}[baseline]
  \matrix(m)[matrix of math nodes,
  column sep = 4em,
  ampersand replacement=\&]
  { - \& \ast \& + \\};
  \path[-stealth,bend left=15]
  (m-1-1.east |- m-1-2.north) edge node[above] {$a_-$} (m-1-2.west |- m-1-2.north)
  (m-1-2.west |- m-1-2.south) edge node[below] {$b_-$} (m-1-1.east |- m-1-2.south);
  \path[-stealth,bend right=15]
  (m-1-3.west |- m-1-2.north) edge node[above] {$a_+$} (m-1-2.east |- m-1-2.north)
  (m-1-2.east |- m-1-2.south) edge node[below] {$b_+$} (m-1-3.west |- m-1-2.south);
  \end{tikzpicture}
\tx{,}
\qquad
  a_{-} b_{-} = a_{+} b_{+}
\tx{.}
\end{gather}

\begin{remark}
The visualization of the Gelfand quiver in~\eqref{E:GelfandQuiver} is the usual one, but in this work another one naturally emerges. We will apply the following identification without further mentioning:
\begin{gather*}
  \left[
  \begin{tikzpicture}[baseline]
  \matrix(m)[matrix of math nodes,
  column sep = 4em,
  ampersand replacement=\&]
  { - \& \ast \& + \\};
  \path[-stealth,bend left=15]
  (m-1-1.east |- m-1-2.north) edge node[above] {$a_-$} (m-1-2.west |- m-1-2.north)
  (m-1-2.west |- m-1-2.south) edge node[below] {$b_-$} (m-1-1.east |- m-1-2.south);
  \path[-stealth,bend right=15]
  (m-1-3.west |- m-1-2.north) edge node[above] {$a_+$} (m-1-2.east |- m-1-2.north)
  (m-1-2.east |- m-1-2.south) edge node[below] {$b_+$} (m-1-3.west |- m-1-2.south);
  \end{tikzpicture}
  \right]
=
  \left[
  \begin{tikzpicture}[baseline]
  \matrix(m)[matrix of math nodes,
  column sep = 4em,
  ampersand replacement=\&]
  { - \& \ast \& + \\};
  \path[-stealth,bend left=15]
  (m-1-1.east |- m-1-2.north) edge node[above] {$a_-$} (m-1-2.west |- m-1-2.north)
  (m-1-2.west |- m-1-2.south) edge node[below] {$b_-$} (m-1-1.east |- m-1-2.south);
  \path[-stealth,bend left=15]
  (m-1-3.west |- m-1-2.south) edge node[below] {$a_+$} (m-1-2.east |- m-1-2.south)
  (m-1-2.east |- m-1-2.north) edge node[above] {$b_+$} (m-1-3.west |- m-1-2.north);
  \end{tikzpicture}
  \right]
\tx{.}
\end{gather*}
\end{remark}

\begin{example}
The isomorphisms between~$\frakB$ and~$\frakA$ and the completed path algebras of~\eqref{E:CyclicQuiver} and~\eqref{E:GelfandQuiver} can be implemented as follows:
\begin{gather*}
  a \mto \begin{psmatrix} 0 & 0 \\ 1 & 0 \end{psmatrix}
\tx{,}\ 
  b \mto \begin{psmatrix} 0 & t \\ 0 & 0 \end{psmatrix}
\tx{;}\quad
  a_- \mto \begin{psmatrix} 0 & t & 0 \\ 0 & 0 & 0 \\ 0 & 0 & 0 \end{psmatrix}
\tx{,}\ 
  a_+ \mto \begin{psmatrix} 0 & 0 & t \\ 0 & 0 & 0 \\ 0 & 0 & 0 \end{psmatrix}
\tx{,}\ 
  b_- \mto \begin{psmatrix} 0 & 0 & 0 \\ 1 & 0 & 0 \\ 0 & 0 & 0 \end{psmatrix}
\tx{,}\ 
  b_+ \mto \begin{psmatrix} 0 & 0 & 0 \\ 0 & 0 & 0 \\ 1 & 0 & 0 \end{psmatrix}
\tx{.}
\end{gather*}
More graphically, a matrix with a single nonzero, monomial entry corresponds to a path, whose source and target are recorded by its position. The column position records the source and the row position records the target of a path. In the case of the cyclic quiver we assign the labels~``$-$'' and~``$+$'' to the first and second position, respectively. The exponent of~$t$ in a monomial records the number of times a path arrives at~``$-$''.
For instance, we next depict the path starting at~``$+$`` and ending at~``$-$`` with~$2$ loops and the reverse path together with the corresponding elements of~$\frakB$:
\begin{gather*}
  \begin{tikzpicture}[baseline]
  \matrix(m)[matrix of math nodes,
  column sep = 4em,
  ampersand replacement=\&]
  { - \& + \\};
  \path[-,bend left=15]
  (m-1-1.east |- m-1-1) edge (m-1-2.west |- m-1-1)
  (m-1-2.west |- m-1-1) edge (m-1-1.east |- m-1-1);
  \path[-,bend left=40]
  (m-1-1.east |- m-1-1) edge (m-1-2.west |- m-1-1)
  (m-1-2.west |- m-1-1) edge (m-1-1.east |- m-1-1);
  \path[-stealth,bend left=80]
  (m-1-2.west |- m-1-1) edge (m-1-1.east |- m-1-1);
  \end{tikzpicture}
\quad
  \begin{psmatrix} 0 & t^3 \\ 0 & 0 \end{psmatrix}
\tx{,}
\qquad
  \begin{tikzpicture}[baseline]
  \matrix(m)[matrix of math nodes,
  column sep = 4em,
  ampersand replacement=\&]
  { - \& + \\};
  \path[-,bend left=15]
  (m-1-1.east |- m-1-1) edge (m-1-2.west |- m-1-1)
  (m-1-2.west |- m-1-1) edge (m-1-1.east |- m-1-1);
  \path[-,bend left=40]
  (m-1-1.east |- m-1-1) edge (m-1-2.west |- m-1-1)
  (m-1-2.west |- m-1-1) edge (m-1-1.east |- m-1-1);
  \path[-stealth,bend left=80]
  (m-1-1.east |- m-1-1) edge (m-1-2.west |- m-1-1);
  \end{tikzpicture}
\quad
  \begin{psmatrix} 0 & 0 \\ t^2 & 0 \end{psmatrix}
\tx{.}
\end{gather*}
The composition of the left with the right path yields five loops starting at~``$+$'', corresponding to~$\begin{psmatrix} 0 & t^3 \\ 0 & 0 \end{psmatrix} \begin{psmatrix} 0 & 0 \\ t^2 & 0 \end{psmatrix} = \begin{psmatrix} 0 & 0 \\ 0 & t^5  \end{psmatrix} \in \frakB$.

In the case of the Gelfand quiver, the source and target~``$\ast$'', ``$-$'', and~``$+$'' in that order correspond to the positions in the matrix. The exponent of~$t$ records the number of times a path arrives at~``$\ast$''. It is crucial for this correspondence that the two loops starting at~``$\ast$'' are considered equivalent by the relation~$a_{-} b_{-} = a_{+} b_{+}$ in~\eqref{E:GelfandQuiver}.
\end{example}

In what follows, $\Rep(\dA)$ and~$\Rep(\dB)$ denote the categories of finite dimensional left $\dA$--modules and $\dB$--modules. They are equivalent to the categories of finite dimensional \emph{nilpotent} representations of~\eqref{E:GelfandQuiver} and~\eqref{E:CyclicQuiver}.

\paragraph{The cyclic quiver}

Let~$M \in \HC_0(\lieg, K)$, whose diagram description is~\eqref{E:DiagramHCmodule}. We use the following notation:
\begin{gather}
\label{eq:def:XYrestrictions_cyclic}
  Z_- = X \big|_{M_{-1}}
\quad\tx{and}\quad
  Z_+ = Y \big|_{M_{+1}}
\tx{.}
\end{gather}

\begin{proposition}
\label{P:Equivalence}
There is an equivalence of categories
\begin{gather}
  \HC_0(\lieg, K) \xlongrightarrow{\EE} \Rep(\dB)
\end{gather}
given on the level of objects by the assignment
\begin{gather}\label{E:FunctorTwoCycles}
  M \xlongmapsto{\EE}
  \left[
  \begin{tikzpicture}[baseline=-0.1\baselineskip]
  \matrix(m)[matrix of math nodes,
  column sep = 4em,
  ampersand replacement=\&]
  { V_- \& V_+ \\};
  \path[-stealth,bend left=15]
  (m-1-1.east |- m-1-1.north) edge node[above] {$A$} (m-1-2.west |- m-1-1.north)
  (m-1-2.west |- m-1-1.south) edge node[below] {$B$} (m-1-1.east |- m-1-1.south);
  \end{tikzpicture}
  \right]
=
  \left[
  \begin{tikzpicture}[baseline=-0.1\baselineskip]
  \matrix(m)[matrix of math nodes,
  column sep = 4em,
  ampersand replacement=\&]
  { M_{-1} \& M_{+1} \\};
  \path[-stealth,bend left=15]
  (m-1-1.east |- m-1-1.north) edge node[above] {$Z_-$} (m-1-2.west |- m-1-1.north)
  (m-1-2.west |- m-1-1.south) edge node[below] {$Z_+$} (m-1-1.east |- m-1-1.south);
  \end{tikzpicture}
  \right]
\tx{.}
\end{gather}
\end{proposition}

\begin{proof}
Let~$M$ and~$N$ be a pair of objects of~$\HC_0(\lieg, K)$ and~$M \xlongrightarrow{\psi} N$ be any morphism. Commutativity of the diagram \eqref{E:MorphismOfReprs} for $p = 0$ implies that the assignment \eqref{E:FunctorTwoCycles} is functorial. Moreover, for any $n \in \ZZ_{> 0}$ we have:
\begin{gather}\label{E:MapsComponents}
\psi_{2n+1} = \bigl(X_N\bigr)^n \psi_1  \bigl(X_M\bigr)^{-n} \quad \tx{and} \quad \psi_{-2n-1} = \bigl(X_N\bigr)^{-n} \psi_{-1} \bigl(X_M\bigr)^{n}.
\end{gather}
Hence, $\psi_1 = 0 = \psi_{-1}$ implies that $\psi = 0$. As a consequence, the functor $\EE$ is faithful. 

Now we prove that $\EE$ is full. Let $\phi_{\pm 1}: M_{\pm 1} \rightarrow N_{\pm 1}$ be linear maps such that 
the diagram
\begin{gather}\label{E:FragmentMorphism}
  \begin{tikzpicture}[baseline]
  \matrix(m)[matrix of math nodes,
  column sep = 4em, row sep = 4em,
  ampersand replacement=\&]
  {  M_{-1} \& M_{+1} 
  \\ N_{-1} \& N_{+1} 
  \\};
  \path[-stealth]
  (m-1-1) edge node[left]  {\small $\phi_{-1}$} (m-2-1)
  (m-1-2) edge node[right] {\small $\phi_{+1}$} (m-2-2);
  \path[-stealth,bend left=10]
  (m-1-1.east |- m-1-1.north) edge node[above] {\small $X_M$} (m-1-2.west |- m-1-1.north)
  (m-1-2.west |- m-1-1.south) edge node[below] {\small $Y_M$} (m-1-1.east |- m-1-1.south);
  \path[-stealth,bend left=10]
  (m-2-1.east |- m-2-1.north) edge node[above] {\small $X_N$} (m-2-2.west |- m-2-1.north)
  (m-2-2.west |- m-2-1.south) edge node[below] {\small $Y_N$} (m-2-1.east |- m-2-1.south);
  \end{tikzpicture}
\end{gather}
is commutative. For any $n \in \ZZ_{>0}$ consider the linear maps $M_{\pm 2n \pm 1} \xlongrightarrow{\phi_{\pm 2n \pm 1}} N_{\pm 2n \pm 1}$ given by the formulae 
$\phi_{2n+1} = \bigl(X_N\bigr)^n \phi_1  \bigl(X_M\bigr)^{-n}$ and $\phi_{-2n-1} = \bigl(X_N\bigr)^{-n} \phi_{-1} \bigl(X_M\bigr)^{n}$. It follows by construction that $\phi_{2n+1} X_M = X_N \phi_{2n-1}$ for all $n \in \ZZ$. 
It can be be checked  that $\phi_{2n-1} Y_M = Y_N \phi_{2n+1}$ for all $n \in \ZZ$, too. We put: 
$\phi := \oplus_{n \in \ZZ} \phi_{2n+1}$. Then $M \xlongrightarrow{\phi} N$ is a morphism in $\HC_0(\lieg, K)$. Hence,  the functor $\EE$ is full, as asserted. 

It remains to be proven that $\EE$ is essentially surjective. Take an arbitrary nilpotent representation 
\begin{gather*}
W = \left[
  \begin{tikzpicture}[baseline=-0.1\baselineskip]
  \matrix(m)[matrix of math nodes,
  column sep = 4em,
  ampersand replacement=\&]
  { V_- \& V_+ \\};
  \path[-stealth,bend left=15]
  (m-1-1.east |- m-1-1.north) edge node[above] {$A$} (m-1-2.west |- m-1-1.north)
  (m-1-2.west |- m-1-1.south) edge node[below] {$B$} (m-1-1.east |- m-1-1.south);
  \end{tikzpicture}
  \right]
\end{gather*}
of the cyclic quiver \eqref{E:CyclicQuiver}. For any $n \in \ZZ_{> 0}$ we put: $M_{2n+1} = V_+$ and 
$M_{-2n-1} = V_-$. Let $M = \oplus_{n \in \ZZ} M_{2n+1}$. We define $X, Y, H \in \End_{\CC}(M)$ using the following rules: 
\begin{itemize}
\item $H\big|_{M_{2n+1}} := (2n+1) \mathsf{Id}$, $X(M_{2n-1}) \subseteq M_{2n+1}$ and $Y(M_{2n+1}) \subseteq M_{2n-1}$ for all $n \in \ZZ$.
\item $X\big|_{M_{-1}} := A$ and $Y\big|_{M_{1}} := B$.
\item $X\big|_{M_{2n-1}} := \mathsf{Id}$ for all $n \ne 0$. 
\end{itemize}
A description of  $Y$ requires more efforts. We first introduce  operators $C_\pm \in \End_{\CC}(V_\pm)$ by the formulae 
\begin{equation}\label{E:Casimirs}
AB = \frac{1}{4}\bigl(C_+ +I\bigr) \quad \tx{and} \quad BA = \frac{1}{4}\bigl(C_- +I\bigr).
\end{equation}
Then for any $n \in \ZZ_{>0}$ we define  the corresponding linear map $$M_{p+1} = V_+ \xlongrightarrow{Y} V_+ = M_{p-1}$$ by the rule: 
$Y = \frac{1}{4}\left(C_+ - (p^2-1)I\right)$, where $p = 2n$.  Analogously, for any $n \in \ZZ_{ < 0}$ and $p = 2n$ we set: 
$$Y := \frac{1}{4}\left(C_- - (p^2-1)I\right): \quad M_{p+1} = V_- \lar V_- = M_{p-1}.$$
It can be checked that $X, Y, H \in \End_{\CC}(M)$ satisfy the identities \eqref{E:DefiningIdentities} and hence define the structure of a $\lieg$-module on the vector space $M$. 

Let $(u_1^\pm, \dots, u_{m_\pm}^\pm)$ be bases of $V_\pm$. Then $u_1^+, \dots, u_{m_+}^+, u_1^-, \dots, u_{m_-}^- \in M$ generate $M$ as a $\lieg$-module. Finally, it can be shown that the Casimir element $C$ defined by \eqref{eq:def:casimir_element} acts on $M_{2n +1} = V_+$ as $C_+$ and on $M_{-2n - 1} = V_-$ as $C_-$ for any $n \in \ZZ_{\ge 0}$. It follows that $(C + I)^m M = 0$, where $m = \max\{m_+, m_-\}$. Summing up, 
$M \in \HC_0(\lieg, K)$ and $\EE(M) \cong W$. Hence, the functor $\EE$ is fully faithful and essentially surjective, hence an equivalence of categories. 
\end{proof}

\paragraph{The Gelfand quiver}

Let~$l \in \ZZ_{> 0}$ and~$M \in \HC_l(\lieg, K)$, whose diagram description is~\eqref{E:DiagramHCmodule}. We use the following notation:
\begin{gather}
\label{eq:def:XYrestrictions}
\begin{aligned}
&
\begin{alignedat}{3}
&
  \Xm = X \big|_{M_{-l-1}}
\tx{,}\quad
&&
  \Xp = X \big|_{M_{l-1}}
\quad\tx{and}\quad
&&
  X_i = X \big|_{M_{-l-1 + 2i}} \tx{\ for\ }1 \le i \le l-1
\tx{;}\\
&
  \Ym = Y \big|_{M_{-l+1}}
\tx{,}\quad
&&
  \Yp = Y \big|_{M_{l+1}}
\quad\tx{and}\quad
&&
  Y_i = Y \big|_{M_{-l + 1 + 2i}} \tx{\ for\ }1 \le i \le l-1
\tx{;}
\end{alignedat}
\\&
  \Xz = X_{l-1} \cdots X_1
\quad\tx{and}\quad
  \Yz = Y_1 \cdots Y_{l-1}
\tx{.}
\end{aligned}
\end{gather}

\begin{lemma}
\label{L:UsefulEquality}
In the above notation~\eqref{eq:def:XYrestrictions}, the following identities are true:
\begin{gather*}
  \Xz \Xm \Ym = \Yp \Xp \Xz
\quad\tx{and}\quad
  \Xm \Ym \Yz = \Yz \Yp \Xp
\tx{.}
\end{gather*}
\end{lemma}

\begin{proof}
We only show  the first statement since a proof of the second one is completely analogous. 
The commutator relation
\begin{gather*}
  \big[ X^l, Y \big] = l X^{l-1} (H + l - 1)
\end{gather*}
implies:
\begin{gather*}
  X_\ast X_- Y_-
=
  X^l Y \big|_{M_{-l+1}}
=
  Y X^l \big|_{M_{-l+1}}
  +
  l X^{l-1} (H + l - 1) \big|_{M_{-l+1}}
=
  Y X^l \big|_{M_{-l+1}}
=
  Y_+ X_+ X_\ast
\tx{.}
\end{gather*}
\end{proof}

\begin{lemma}
\label{la:nilpotent_gelfand_quiver_representations}
Let~$l \in \ZZ_{> 0}$ and~$M$ be an object of\/~$\HC_l(\lieg, K)$ with diagram description
\begin{gather}
\label{E:DiagramHCmoduleI}
  \begin{tikzpicture}[baseline]
  \matrix(m)[matrix of math nodes,
  column sep = 2.5em,
  ampersand replacement=\&]
  {  \cdots \& M_{-l-1} \& M_{-l+1}
  \& \cdots \& M_{l-1}  \& M_{l+1}
  \& \cdots \\};
  \path[-stealth,bend left=15]
  (m-1-1.east |- m-1-2.north) edge                         (m-1-2.west |- m-1-2.north)
  (m-1-2.east |- m-1-2.north) edge node[above] {$\Xm$}     (m-1-3.west |- m-1-2.north)
  (m-1-3.east |- m-1-2.north) edge node[above] {$X_1$}     (m-1-4.west |- m-1-2.north)
  (m-1-4.east |- m-1-2.north) edge node[above] {$X_{l-1}$} (m-1-5.west |- m-1-2.north)
  (m-1-5.east |- m-1-2.north) edge node[above] {$\Xp$}     (m-1-6.west |- m-1-2.north)
  (m-1-6.east |- m-1-2.north) edge                         (m-1-7.west |- m-1-2.north)
  (m-1-7.west |- m-1-2.south) edge                         (m-1-6.east |- m-1-2.south)
  (m-1-6.west |- m-1-2.south) edge node[below] {$\Yp$}     (m-1-5.east |- m-1-2.south)
  (m-1-5.west |- m-1-2.south) edge node[below] {$Y_{l-1}$} (m-1-4.east |- m-1-2.south)
  (m-1-4.west |- m-1-2.south) edge node[below] {$Y_1$}     (m-1-3.east |- m-1-2.south)
  (m-1-3.west |- m-1-2.south) edge node[below] {$\Ym$}     (m-1-2.east |- m-1-2.south)
  (m-1-2.west |- m-1-2.south) edge                         (m-1-1.east |- m-1-2.south);
  \end{tikzpicture}
\end{gather}
Then
\begin{gather}
\label{E:PairofRepsgelfand}
  \begin{tikzpicture}[baseline]
  \matrix(m)[matrix of math nodes,
  column sep = 2.5em,
  ampersand replacement=\&]
  { M_{-l-1} \& M_{-l+1} \& M_{l+1} \\};
  \path[-stealth,bend left=15]
  (m-1-1.east |- m-1-1.north) edge node[above] {$\Xm$}          (m-1-2.west |- m-1-1.north)
  (m-1-2.east |- m-1-1.north) edge node[above] {$\Xp \Xz$}      (m-1-3.west |- m-1-1.north)
  (m-1-3.west |- m-1-1.south) edge node[below] {$\Xz^{-1} \Yp$} (m-1-2.east |- m-1-1.south)
  (m-1-2.west |- m-1-1.south) edge node[below] {$\Ym$}          (m-1-1.east |- m-1-1.south);
  \end{tikzpicture}
\quad\tx{and}\quad 
  \begin{tikzpicture}[baseline]
  \matrix(m)[matrix of math nodes,
  column sep = 2.5em,
  ampersand replacement=\&]
  { M_{-l-1} \& M_{l-1} \& M_{l+1} \\};
  \path[-stealth,bend left=15]
  (m-1-1.east |- m-1-1.north) edge node[above] {$\Yz^{-1} \Xm$} (m-1-2.west |- m-1-1.north)
  (m-1-2.east |- m-1-1.north) edge node[above] {$\Xp$}          (m-1-3.west |- m-1-1.north)
  (m-1-3.west |- m-1-1.south) edge node[below] {$\Yp$}          (m-1-2.east |- m-1-1.south)
  (m-1-2.west |- m-1-1.south) edge node[below] {$\Ym \Yz$}      (m-1-1.east |- m-1-1.south);
  \end{tikzpicture}
\end{gather}
are nilpotent representations of the Gelfand quiver~\eqref{E:GelfandQuiver} which are moreover isomorphic (as before, $\Xz = X_{l-1} \cdots X_1$ and 
$\Yz = Y_1 \cdots Y_{l-1}$).
\end{lemma}

\begin{proof}
The fact that~\eqref{E:PairofRepsgelfand} are representations of the Gelfand quiver follows from Lemma~\ref{L:UsefulEquality}. The requested isomorphism of representations is constructed  as follows:
\begin{gather}
\label{E:MorphismOfReps}
  \begin{tikzpicture}[baseline]
  \matrix(m)[matrix of math nodes,
  column sep = 4em, row sep = 4em,
  ampersand replacement=\&]
  {  M_{-l-1} \& M_{-l+1} \& M_{l+1}
  \\ M_{-l-1} \& M_{l-1}  \& M_{l+1}
  \\};
  \path[-stealth]
  (m-1-1) edge node[left]  {\small $T$}   (m-2-1)
  (m-1-2) edge node[right] {\small $\Xz$} (m-2-2)
  (m-1-3) edge node[right] {\small $I$}   (m-2-3);
  \path[-stealth,bend left=10]
  (m-1-1.east |- m-1-1.north) edge node[above] {\small $\Xm$}          (m-1-2.west |- m-1-1.north)
  (m-1-2.east |- m-1-1.north) edge node[above] {\small $\Xp \Xz$}      (m-1-3.west |- m-1-1.north)
  (m-1-3.west |- m-1-1.south) edge node[below] {\small $\Xz^{-1} \Yp$} (m-1-2.east |- m-1-1.south)
  (m-1-2.west |- m-1-1.south) edge node[below] {\small $\Ym$}          (m-1-1.east |- m-1-1.south);
  \path[-stealth,bend left=10]
  (m-2-1.east |- m-2-1.north) edge node[above] {\small $\Yz^{-1} \Xm$} (m-2-2.west |- m-2-1.north)
  (m-2-2.east |- m-2-1.north) edge node[above] {\small $\Xp$}          (m-2-3.west |- m-2-1.north)
  (m-2-3.west |- m-2-1.south) edge node[below] {\small $\Yp$}          (m-2-2.east |- m-2-1.south)
  (m-2-2.west |- m-2-1.south) edge node[below] {\small $\Ym \Yz$}      (m-2-1.east |- m-2-1.south);
  \end{tikzpicture}
\end{gather}
First note that both squares in the right part of~\eqref{E:MorphismOfReps} are commutative. To define $T$, note that
\begin{gather*}
  \Yz \Xz = Y_1 \dots Y_{l-1} X_{l-1} \dots X_1 = p(C_1)
\tx{,}
\end{gather*}
where $p(t) = \frac{1}{4^{l-1}}(t - \gamma_1) \dots (t-\gamma_{l-1}) \in \CC[t]$ for appropriate $\gamma_1, \dots, \gamma_{l-1} \in \ZZ$, which are  all different from $\gamma$. We put $T = p(C_0)$, where $C_0 = C\big|_{M_{-l-1}}$. It follows that $T$ is an isomorphism of vector spaces. The commutativity  of both left squares of the diagram~\eqref{E:MorphismOfReps} follows in particular from the fact that $C$ is a central element of $U(\lieg)$.
\end{proof}

\begin{theorem}%
\label{T:Equivalence}
For any~$l \in \ZZ_{> 0}$ there is an equivalence of categories
\begin{gather}
  \HC_l(\lieg, K) \xlongrightarrow{\EE} \Rep(\dA)
\end{gather}
given on the level of objects by the assignment
\begin{gather}\label{E:AssignmentGelfand}
  M \xlongmapsto{\EE}
  \left[
  \begin{tikzpicture}[baseline]
  \matrix(m)[matrix of math nodes,
  column sep = 2.5em,
  ampersand replacement=\&]
  { V_- \& V_\ast \& V_+ \\};
  \path[-stealth,bend left=15]
  (m-1-1.east |- m-1-2.north) edge node[above] {$A_-$} (m-1-2.west |- m-1-2.north)
  (m-1-2.west |- m-1-2.south) edge node[below] {$B_-$} (m-1-1.east |- m-1-2.south);
  \path[-stealth,bend right=15]
  (m-1-3.west |- m-1-2.north) edge node[above] {$A_+$} (m-1-2.east |- m-1-2.north)
  (m-1-2.east |- m-1-2.south) edge node[below] {$B_+$} (m-1-3.west |- m-1-2.south);
  \end{tikzpicture}
  \right]
= 
  \left[
  \begin{tikzpicture}[baseline]
  \matrix(m)[matrix of math nodes,
  column sep = 2.5em,
  ampersand replacement=\&]
  { M_{-l-1} \& M_{-l+1} \& M_{l+1} \\};
  \path[-stealth,bend left=15]
  (m-1-1.east |- m-1-1.north) edge node[above] {$\Xm$}          (m-1-2.west |- m-1-1.north)
  (m-1-2.east |- m-1-1.north) edge node[above] {$\Xp \Xz$}      (m-1-3.west |- m-1-1.north)
  (m-1-3.west |- m-1-1.south) edge node[below] {$\Xz^{-1} \Yp$} (m-1-2.east |- m-1-1.south)
  (m-1-2.west |- m-1-1.south) edge node[below] {$\Ym$}          (m-1-1.east |- m-1-1.south);
  \end{tikzpicture}
  \right]
\tx{.}
\end{gather}
\end{theorem}

\begin{proof} Let $M \xlongrightarrow{\psi} N$ be a morphism in $\HC_l(\lieg, K)$. We claim that the following diagram 
\begin{gather}
\label{E:MorphismOfRepsN}
  \begin{tikzpicture}[baseline]
  \matrix(m)[matrix of math nodes,
  column sep = 4em, row sep = 4em,
  ampersand replacement=\&]
  {  M_{-l-1} \& M_{-l+1} \& M_{l+1}
  \\ N_{-l-1} \& N_{l-1}  \& N_{l+1}
  \\};
  \path[-stealth]
  (m-1-1) edge node[left]  {\small $\psi_{-l-1}$}   (m-2-1)
  (m-1-2) edge node[left] {\small $\psi_{-l+1}$} (m-2-2)
  (m-1-3) edge node[right] {\small $\psi_{l+1}$}   (m-2-3);
  \path[-stealth,bend left=10]
  (m-1-1.east |- m-1-1.north) edge node[above] {\small $\Xm$}          (m-1-2.west |- m-1-1.north)
  (m-1-2.east |- m-1-1.north) edge node[above] {\small $\Xp \Xz$}      (m-1-3.west |- m-1-1.north)
  (m-1-3.west |- m-1-1.south) edge node[below] {\small $\Xz^{-1} \Yp$} (m-1-2.east |- m-1-1.south)
  (m-1-2.west |- m-1-1.south) edge node[below] {\small $\Ym$}          (m-1-1.east |- m-1-1.south);
  \path[-stealth,bend left=10]
  (m-2-1.east |- m-2-1.north) edge node[above] {\small $\Xm$} (m-2-2.west |- m-2-1.north)
  (m-2-2.east |- m-2-1.north) edge node[above] {\small $\Xp \Xz$}          (m-2-3.west |- m-2-1.north)
  (m-2-3.west |- m-2-1.south) edge node[below] {\small $\Xz^{-1} \Yp$}          (m-2-2.east |- m-2-1.south)
  (m-2-2.west |- m-2-1.south) edge node[below] {\small $\Ym $}      (m-2-1.east |- m-2-1.south);
  \end{tikzpicture}
\end{gather}
is commutative (abusing the notation, we use same symbols for the structure maps of $\lieg$-modules $M$ and $N$ in their diagram descriptions). It follows from \eqref{E:MorphismOfReprs}  that the left square of \eqref{E:MorphismOfRepsN} is commutative and $\psi_{l+1} X_+ X_\ast = X_+ X_\ast \psi_{-l+1}$. Next, $\psi_{l-1} X_\ast = X_\ast \psi_{-l+1}$ and $\psi_{l-1} Y_+ = Y_+ \psi_{l+1}$, implying that 
$$
\psi_{-l+1} X_\ast^{-1} Y_+ = X_\ast^{-1} \psi_{l-1} Y_+ = X_\ast^{-1}  Y_+ \psi_{l+1}.
$$
Hence, the diagram \eqref{E:MorphismOfRepsN} is indeed commutative, which implies that the assignment \eqref{E:AssignmentGelfand} is indeed functorial.

Next, similarly to the case $l = 0$ (see \eqref{E:MapsComponents}) one can show that all    linear maps $\psi_{-l+1+2n}$ can be expressed via 
$\psi_{-l-1}$, $\psi_{-l+1}$, $\psi_{l+1}$ and the structure maps $X$ and $Y$ of the modules $M$ and $N$. This implies that the functor $\EE$ is faithful. The fact that 
$\EE$ is full and essentially surjective can be shown analogously  to the proof of Proposition \ref{P:Equivalence}. See also \cite[Section 3.9]{Mazorchuk2010} for another description of an equivalence between $\HC_l(\lieg, K)$ and  $\Rep(\dA)$.
\end{proof}

\begin{remark}%
\label{R:SecondDescriptionOfTheEquivalence}

In what follows we shall use the fact that in the spirit of Lemma~\ref{la:nilpotent_gelfand_quiver_representations} 
\begin{gather}
  \EE(M)
\cong
  \begin{tikzpicture}[baseline]
  \matrix(m)[matrix of math nodes,
  column sep = 2.5em,
  ampersand replacement=\&]
  { M_{-l-1} \& M_{l-1} \& M_{l+1}. \\};
  \path[-stealth,bend left=15]
  (m-1-1.east |- m-1-1.north) edge node[above] {$\Yz^{-1} \Xm$} (m-1-2.west |- m-1-1.north)
  (m-1-2.east |- m-1-1.north) edge node[above] {$\Xp$}          (m-1-3.west |- m-1-1.north)
  (m-1-3.west |- m-1-1.south) edge node[below] {$\Yp$}          (m-1-2.east |- m-1-1.south)
  (m-1-2.west |- m-1-1.south) edge node[below] {$\Ym \Yz$}      (m-1-1.east |- m-1-1.south);
  \end{tikzpicture}
\end{gather}
\end{remark}

\section{Cyclic modules over \texpdf{$\dA$}{A} and \texpdf{$\dB$}{B}}
\label{sec:cyclic_modules}

We refer to~\cite{CurtisReiner} for basic notions and results from the representation theory of associative algebras. 
Let
\begin{gather}\label{E:Idempotents}
  e_\ast
=
  \begin{psmatrix}
  1 & 0 & 0 \\
  0 & 0 & 0 \\
  0 & 0 & 0
  \end{psmatrix}
\tx{,} \; 
  e_+
=
  \begin{psmatrix}
  0 & 0 & 0 \\
  0 & 1 & 0 \\
  0 & 0 & 0
  \end{psmatrix}
\; \tx{and\ } \;
  e_-
=
  \begin{psmatrix}
  0 & 0 & 0 \\
  0 & 0 & 0 \\
  0 & 0 & 1
  \end{psmatrix}
\end{gather}
be three primitive idempotents of the algebra $\dA$ corresponding to three vertices of the Gelfand quiver~\eqref{E:GelfandQuiver}. Then
\begin{gather*}
  P_\ast = \dA e_\ast
\cong 
  \begin{psmatrix}
  \dD \\
  \dD \\
  \dD
  \end{psmatrix}
\tx{,} \; 
  P_+ = \dA e_+
\cong
  \begin{psmatrix}
  \idm \\
  \dD \\
  \idm
  \end{psmatrix}
\; \tx{and\ } \; 
  P_- = \dA e_-
\cong
  \begin{psmatrix}
  \idm \\
  \idm \\
  \dD
  \end{psmatrix}
\end{gather*}
are indecomposable projective $\dA$--modules. Conversely, any indecomposable projective $\dA$--module is isomorphic to $P_\ast$ or $P_\pm$. 

\begin{definition}
A finite dimensional $\dA$--module $V$ is called \emph{cyclic} if there exists an epimorphism $P \thra V$ for an indecomposable projective $\dA$--module $P$.
\end{definition}

\begin{lemma}
Any  cyclic\/~$\dA$--module is indecomposable.
\end{lemma}

\begin{proof}
Assume that $V \cong V' \oplus V''$ is a direct sum decomposition with $V' \ne 0 \ne V''$. Take projective coverings $P' \xtwoheadrightarrow{f'} V'$ and $P'' \xtwoheadrightarrow{f''} V''$. Then $P' \oplus P'' \xtwoheadrightarrow{(f'\, f'')} V$ is a projective covering of $V$. It follows that $P \cong P' \oplus P''$ is decomposable. Contradiction. 
\end{proof}

\begin{remark}
The above result is true for arbitrary semi-perfect rings (i.e.\@ those rings for which any finitely generated left module has a projective cover).
\end{remark}

Our next goal is to give a classification of cyclic $\dA$--modules. 

\begin{definition}
A finitely generated $\dA$--module $Q$ is called a \emph{lattice} if it is free as a $\dD$--module. 
\end{definition}

The algebra $\dA$ belongs to the class of the so-called nodal orders; see~\cite{BurbanDrozdNodal} and in particular~\cite[Example 3.18]{BurbanDrozdNodal}. The study of lattices over orders is a classical subject of representation theory of associative algebras~\cite{CurtisReiner}. Consider the so-called hereditary envelope
\begin{gather}
  \mathfrak{H}
= 
  \begin{psmatrix}
  \dD & \idm & \idm\\
  \dD & \dD & \dD \\
  \dD & \dD & \dD
  \end{psmatrix}
\end{gather}
of the nodal order $\dA$ (see~\cite{BurbanDrozdNodal} for the definition). Then any indecomposable $\dA$--lattice is a direct summand of $\dA \oplus \mathfrak{H}$ (this result is true for arbitrary nodal orders; see~\cite{BurbanDrozdNodal}). As a consequence, there are only four isomorphism classes of  indecomposable $\dA$--lattices: indecomposable projective modules $P_\ast, P_+$ and $P_-$ and as well as
\begin{gather*}
  Q
= 
  \begin{psmatrix}
  \idm \\
  \dD \\
  \dD
  \end{psmatrix}
\tx{.}
\end{gather*}
This fact is essential for the classification of cyclic $\dA$--modules given below.

For a finite dimensional $\dA$--module
\begin{gather}%
\label{E:GelfandQuiverRepres}
  V
=
  \left[
  \begin{tikzpicture}[baseline]
  \matrix(m)[matrix of math nodes,
  column sep = 4em,
  ampersand replacement=\&]
  { V_- \& V_\ast \& V_+ \\};
  \path[-stealth,bend left=15]
  (m-1-1.east |- m-1-2.north) edge node[above] {$A_-$} (m-1-2.west |- m-1-2.north)
  (m-1-2.west |- m-1-2.south) edge node[below] {$B_-$} (m-1-1.east |- m-1-2.south);
  \path[-stealth,bend right=15]
  (m-1-3.west |- m-1-2.north) edge node[above] {$A_+$} (m-1-2.east |- m-1-2.north)
  (m-1-2.east |- m-1-2.south) edge node[below] {$B_+$} (m-1-3.west |- m-1-2.south);
  \end{tikzpicture}
  \right]
\end{gather}
we put: $C_\pm := B_\pm A_\pm$ and $C_\ast = A_+ B_+ = A_- B_-$. Let $m_\pm$ and $m_\ast$ be the nilpotency degrees of the endomorphisms $C_\pm$ and $C_\ast$, respectively, i.e.\@ $C_\ast^{m_\ast+1} = 0$ but $C_\ast^{m_\ast} \ne 0$ etc., whereas
\begin{gather*}
  \underline{\dim}(V)
=
  \bigl(\dim(V_-),\, \dim(V_\ast),\, \dim(V_+)\bigr) \in \ZZ_{\ge 0}^3
\end{gather*}
is the dimension vector of $V$. 

A cyclic $\dA$--module $V$ has \emph{type} $i \in \{+, -,  \ast\}$ if there exists an epimorphism $P_i \thra V$.

\begin{theorem}%
\label{T:GelfandQuiverCyclicReps}
Let $V$ be a cyclic $\dA$--module of type $i \in \left\{+, -, \ast\right\}$. Then its isomorphism class is uniquely determined by the corresponding dimension vector $\underline{\dim}(V)$. For $d \in \ZZ_{\ge 0}$, the possibilities are as follows:
\begin{enumerate}[label=(\Roman*)]
\item
\label{it:T:GelfandQuiverCyclicReps:ast}
\Needspace*{5\baselineskip}
Cyclic modules of type~``$\ast$''. Possible dimension vectors are:
\begin{enumerate}[label=(\alph*),itemsep=0pt,parsep=0pt]
\item  $(d, d+1, d)$, 
\item  $(d+1, d+1, d+1)$,  
\item  $(d, d+1, d+1)$,
\item  $(d+1, d+1, d)$.
\end{enumerate}
For all these representations we have: $m_\ast = d+1$. 

\item
\label{it:T:GelfandQuiverCyclicReps:plus}
\Needspace*{5\baselineskip}
Cyclic modules of type~``$+$''. Possible dimension vectors are:
\begin{enumerate}[label=(\alph*),itemsep=0pt,parsep=0pt]
\item  $(d, d, d+1)$.,
\item  $(d, d+1, d+1)$,
\item  $(d+1, d+1, d+1)$,
\item  $(d-1, d, d+1)$ (it exists only for $d\ge 1$).
\end{enumerate}
For these  cyclic modules we have: $m_+ = d+1$.

\item
\label{it:T:GelfandQuiverCyclicReps:minus}
\Needspace*{5\baselineskip}
A description of the cyclic modules of type~``$-$'' is symmetric to those given in~\ref{it:T:GelfandQuiverCyclicReps:plus}.
Specifically, possible dimension vectors are:
\begin{enumerate}[label=(\alph*),itemsep=0pt,parsep=0pt]
\item  $(d+1, d, d)$.,
\item  $(d+1, d+1, d)$,
\item  $(d+1, d+1, d+1)$,
\item  $(d+1, d, d-1)$ (it exists only for $d\ge 1$).
\end{enumerate}
For these  cyclic modules we have: $m_- = d+1$.
\end{enumerate}
\end{theorem}

\begin{proof}
Let $V$ be a cyclic $\dA$--module and $P \xtwoheadrightarrow{\pi} V$ be its projective cover. Consider the short exact sequence
\begin{gather*}
  0 \lar N \xlongrightarrow{\imath} P \xlongrightarrow{\pi} V \lar 0
\tx{.}
\end{gather*}
Note that $P$ is free viewed as a $\dD$--module. 
Since $\dD$ is a principal ideal domain and $N$ is a $\dD$-submodule of $P$, $N$ is free over  $\dD$ as well and  thus is an $\dA$--lattice. As $V$ is a $\dD$--module of finite length, we have: 
\begin{gather*}
  \mathfrak{K} \otimes N
\cong
  \mathfrak{K} \otimes P \cong U
:=
  \begin{psmatrix}
  \mathfrak{K} \\
  \mathfrak{K} \\
  \mathfrak{K}
  \end{psmatrix}
\tx{,}
\end{gather*}
where $\mathfrak{K} = \CC\llbrace t\rrbrace$. Hence, $N$ is an \emph{indecomposable} $\dA$--lattice, i.e.\@ $N \in \big\{P_\ast, P_\pm, Q \big\}$. Next, we can view $P$ and $N$ as \emph{subsets} of $U$. Then we get:
\begin{gather*}
  \Hom_{\dA}(N, P)
=
  \big\{a \in \mathfrak{K} \,:\, a N \subseteq P  \big\}
=
  \big\{a \in \dD \,:\, a N \subseteq P  \big\}
\tx{.}
\end{gather*}
It follows that any non-zero $\dA$--module map $N \ra P$ is injective and its cokernel is a cyclic $\dA$--module.  We hence have the following cases, corresponding to the enumeration in the Theorem~\ref{T:GelfandQuiverCyclicReps}.
\begin{enumerate}[label=(\Roman*)]
\item
Cyclic modules of type ``$\ast$''.
\begin{enumerate}[label=(\alph*)]
\item
\label{it:representation_gelfand:Ia}
The dimension vector of
\begin{gather*}
  \cok\Big(Q \xrightarrow{t^{d}} P_\ast\Big)
= 
  \cok\Big(
  \begin{psmatrix}
  \idm \\
  \dD \\
  \dD
  \end{psmatrix}
  \xrightarrow{t^{d}}
  \begin{psmatrix}
  \dD \\
  \dD \\
  \dD
  \end{psmatrix}
  \Big)
\end{gather*}
is $(d, d+1, d)$. 

\item
\label{it:representation_gelfand:Ib}
The dimension vector of
\begin{gather*}
  \cok\Big(P_\ast \xrightarrow{t^{d+1}} P_\ast\Big)
= 
  \cok\Big(
  \begin{psmatrix}
  \dD \\
  \dD \\
  \dD
  \end{psmatrix}
  \xrightarrow{t^{d+1}}
  \begin{psmatrix}
  \dD \\
  \dD \\
  \dD
  \end{psmatrix}
  \Big)
\end{gather*}
is $(d+1, d+1, d+1)$. 

\item
\label{it:representation_gelfand:Ic}
The dimension vector of
\begin{gather*}
  \cok\Big(P_- \xrightarrow{t^{d}} P_\ast\Big)
= 
  \cok\Big(
  \begin{psmatrix}
  \idm \\
  \idm \\
  \dD
  \end{psmatrix}
  \xrightarrow{t^{d}}
  \begin{psmatrix}
  \dD \\
  \dD \\
  \dD
  \end{psmatrix}
  \Big)
\end{gather*}
is $(d, d+1, d+1)$. 

\item
\label{it:representation_gelfand:Id}
The dimension vector of
\begin{gather*}
  \cok\Big(P_+ \xrightarrow{t^{d}} P_\ast\Big)
= 
  \cok\Big(
  \begin{psmatrix}
  \idm \\
  \dD \\
  \idm
  \end{psmatrix}
  \xrightarrow{t^{d}}
  \begin{psmatrix}
  \dD \\
  \dD \\
  \dD
  \end{psmatrix}
  \Big)
\end{gather*}
is $(d+1, d+1, d)$.
\end{enumerate}

\item
\begin{enumerate}[label=(\alph*)]
\item
\label{it:representation_gelfand:IIa}
The dimension vector of
\begin{gather*}
  \cok\Big(P_\ast \xrightarrow{t^{d+1}} P_+\Big)
= 
  \cok\Big(
  \begin{psmatrix}
  \dD \\
  \dD \\
  \dD
  \end{psmatrix}
  \xrightarrow{t^{d+1}} 
  \begin{psmatrix}
  \idm \\
  \dD \\
  \idm
  \end{psmatrix}
  \Big)
\end{gather*}
is $(d, d, d+1)$. 

\item
\label{it:representation_gelfand:IIb}
The dimension vector of
\begin{gather*}
  \cok\Big(Q \xrightarrow{t^{d+1}} P_+\Big)
= 
  \cok\Big(
  \begin{psmatrix}
  \idm \\
  \dD \\
  \dD
  \end{psmatrix}
  \xrightarrow{t^{d+1}} 
  \begin{psmatrix}
  \idm \\
  \dD \\
  \idm
  \end{psmatrix}
  \Big)
\end{gather*}
is $(d, d+1, d+1)$. 

\item
\label{it:representation_gelfand:IIc}
The dimension vector of
\begin{gather*}
  \cok\Big(P_+ \xrightarrow{t^{d+1}} P_+\Big)
= 
  \cok\Big(
  \begin{psmatrix}
  \idm \\
  \dD \\
  \dD
  \end{psmatrix}
  \xrightarrow{t^{d+1}}
  \begin{psmatrix}
  \idm \\
  \dD \\
  \idm
  \end{psmatrix}
  \Big)
\end{gather*}
is $(d+1, d+1, d+1)$. 

\item
\label{it:representation_gelfand:IId}
The dimension vector of
\begin{gather*}
  \cok\Big(P_- \xrightarrow{t^{d}} P_+\Big)
= 
  \cok\Big(
  \begin{psmatrix}
  \idm \\
  \idm \\
  \dD
  \end{psmatrix}
  \xrightarrow{t^{d}}
  \begin{psmatrix}
  \idm \\
  \dD \\
  \idm
  \end{psmatrix}
  \Big)
\end{gather*}
is $(d-1, d, d+1)$ with $d \ge 1$.
\end{enumerate}
\end{enumerate}
The assertion about the nilpotency degrees $m_\pm$ and $m_\ast$ follows from the observation that  $C_\ast$ (respectively, $C_\pm$) is a nilpotent Jordan block of size $\dim(V_\ast)$ (respectively, $\dim(V_\pm)$). 
\end{proof}

\begin{remark}
Theorem \ref{T:GelfandQuiverCyclicReps} asserts that 
cyclic $\dA$--modules are determined by purely ``discrete data'': their types and dimension vectors. It is not true in general, even in the case of arbitrary nodal orders. For example, let $\mathfrak{N} = \CC\llbracket x, y\rrbracket \big\slash \mspace{-5mu}\big\slash  (xy)$.Then for $n, m \in \ZZ_{> 0}$ and $\lambda \in \CC^\ast$ the module $V_{(m, n, \lambda)} = \mathfrak{N} \big\slash (x^m - \lambda y^n)$ is cyclic. However, $V_{(m, n, \lambda)} \cong V_{(m', n', \lambda')}$ if and only if~$(m, n, \lambda) = (m', n', \lambda')$.
\end{remark}

\begin{remark} 
The problem of the classification of \emph{all} indecomposable finite dimensional $\dA$--modules  was posed by I.~Gelfand in~\cite{Gelfand}. 
In 1973 Nazarova and Roiter proved that $\Rep(\dA)$  is representation tame, reducing the problem of the description of its indecomposable objects to a certain problem of linear algebra (matrix problem). However, the correct combinatorics
of indecomposables was obtained only in 1988 by Bondarenko~\cite{mp1, mp}. Independently and about the same time, Crawley-Boevey gave another solution of Gelfand's problem using a  completely different approach~\cite{CB}. In 2004 Burban and Drozd proved that the derived category $\mathsf{D}^b\bigl(\Rep(\dA)\bigr)$ is also representation tame and gave an explicit description of the corresponding indecomposable complexes of projective modules~\cite{Nodal}. See also~\cite{Gnedin} for further elaborations of this approach. 
\end{remark}

\begin{remark}
\label{rm:CyclicQuiverCyclicReps}
Since $\dB$ is a hereditary order, any indecomposable finite dimensional $\dB$--module is cyclic; see for instance \cite{Drozd}.
Let 
\begin{gather*}
  V
=
  \left[
  \begin{tikzpicture}[baseline]
  \matrix(m)[matrix of math nodes,
  column sep = 4em,
  ampersand replacement=\&]
  { V_- \& V_+ \\};
  \path[-stealth,bend left=15]
  (m-1-1.east |- m-1-1.north) edge node[above] {$A_+$} (m-1-2.west |- m-1-1.north)
  (m-1-2.west |- m-1-1.south) edge node[below] {$A_-$} (m-1-1.east |- m-1-1.south);
  \end{tikzpicture}
  \right]
\end{gather*}
be a nilpotent representation of~\eqref{E:CyclicQuiver}. We put: $C_\pm := A_\mp A_\pm$. Let $m_\pm$ be the nilpotency degree of $C_\pm$ and $\underline{\dim}(V) = \bigl(\dim(V_-), \dim(V_+)\bigr)$ the dimension vector of $V$. Let
\begin{gather*}
  e_+
=
  \begin{psmatrix}
  1 & 0 \\
  0 & 0
  \end{psmatrix}
  \in \dB
\quad\tx{and}\quad 
  e_-
=
  \begin{psmatrix}
  0 & 0 \\
  0 & 1
  \end{psmatrix}
  \in \dB
\tx{.}
\end{gather*}
We put:
\begin{gather*}
  P_+ = \dB e_+
\cong
  \begin{psmatrix}
  \dD \\
  \dD
  \end{psmatrix}
\quad\tx{and}\quad 
  P_- = \dB e_-
\cong
  \begin{psmatrix}
  \idm \\
  \dD
  \end{psmatrix}
\tx{.}
\end{gather*}
Cyclic $\dB$--modules are up to isomorphism characterized by their types and dimension vectors. Let $d \in \ZZ_{\ge 0}$.
\begin{enumerate}[label=(\Roman*)]
\item Cyclic modules of type~``$+$'' are the following.
\begin{enumerate}[label=(\alph*)]
\item
\label{it:representation_cyclic:Ia}
The dimension vector of
\begin{gather*}
  \cok\Big(P_- \xrightarrow{t^{d}} P_+\Big)
= 
  \cok\Big(
  \begin{psmatrix}
  \idm \\
  \dD
  \end{psmatrix}
  \xrightarrow{t^{d}}
  \begin{psmatrix}
  \dD\\
  \dD \\
  \end{psmatrix}
  \Big)
\end{gather*}
is $(d, d+1)$. 

\item
\label{it:representation_cyclic:Ib}
The dimension vector of
\begin{gather*}
  \cok\Big(P_+ \xrightarrow{t^{d}} P_+\Big)
= 
  \cok\Big(
  \begin{psmatrix}
  \idm \\
  \dD
  \end{psmatrix}
  \xrightarrow{t^{d+1}}
  \begin{psmatrix}
  \dD\\
  \dD \\
  \end{psmatrix}
  \Big)
\end{gather*}
is $(d+1, d+1)$.
\end{enumerate}
For both types of cyclic modules we have: $m_+ = d + 1$.

\item The classification of cyclic $\dB$--modules of type~``$-$'' is analogous. 
\end{enumerate}
\end{remark}

\section{From polyharmonic Maa\ss{} forms to quiver representations}%
\label{sec:polyharmonic_to_quiver}

We describe and classify quiver representations associated to polyharmonic weak Maa\ss{} forms and provide in each case of the classification polyharmonic weak Maa\ss{} forms in all possible weights that yield the relevant quiver representation. The classification is done in Section~\ref{ssec:harish_chandra_modules_quiver_representations} by linking it to the classification of cyclic quiver representations from Section~\ref{sec:cyclic_modules}. Our modular realizations are given in Section~\ref{sec:modular_realization}.

The representation theoretic labels that arise in Section~\ref{sec:cyclic_modules} do not match the ones for harmonic Maa\ss{} forms that appeared in work of Bringmann--Kudla~\cite{BK}. We provide a translation in Section~\ref{ssec:classification_BK_labels}.

\subsection{Polyharmonic Maa\ss{} forms}%
\label{ssec:polyharmonic_maass_forms}

The group $G = \SL{2}(\RR)$ acts on the upper half-plane $\HH = \{\tau = x + i y \in \CC \,:\, \Im(\tau) > 0\}$ by M\"obius transformations:
\begin{gather*}
  \begin{psmatrix}
  a & b \\
  c & d
  \end{psmatrix}
  \tau
=
  \mfrac{a\tau + b}{c\tau + d}
\tx{.}
\end{gather*}
For~$k \in \ZZ$, we obtain the weight-$k$ slash action on functions~$f : \HS \ra \CC$ defined by
\begin{gather*}
  \big( f \big|_k g \big)(\tau)
=
  (c \tau + d)^{-k} f(g \tau)
\tx{,}\quad
  g = \begin{psmatrix} a & b \\ c & d \end{psmatrix} \in \SL{2}(\RR)
\tx{.}
\end{gather*}

Given~$k \in \ZZ$, we consider the weight-$k$ \emph{Maa\ss{} lowering and raising operators}
\begin{gather*}
  \rmL_k = - 2 i y^2 \partial_{\ov\tau}
\quad\tx{and}\quad
  \rmR_k = 2 i \partial_\tau + k y^{-1}
\tx{,}
\end{gather*}
and the~\emph{Maa\ss{}--Laplace operator}
\begin{gather}\label{eq:laplace}
  \Delta_k
=
  -y^2 (\partial_x^2 + \partial_y^2) + ik y (\partial_x +i  \partial_y)
=
  -\rmR_{k-2}\, \rmL_k
=
  - (\rmL_{k+2} \rmR_k+k)
\tx{,}
\end{gather}
acting on the space of smooth functions on $\HH$.  We also define iterated versions of the lowering and raising operators by
\begin{gather*}
  \rmL_k^j
=
  \rmL_{k-2(j-1)} \circ\,\cdots\,\circ \rmL_{k-2}\circ \rmL_k
\quad\tx{and}\quad
  \rmR_k^j
=
  \rmR_{k+2(j-1)} \circ\,\cdots\,\circ \rmR_{k+2}\circ \rmR_k
\tx{.}
\end{gather*}

For~$k \in \ZZ_{\le 0}$, we define the \emph{flipping operator}
\begin{gather}
\label{eq:def:flipping_operator}
  \rmF_k\, f
=
  \mfrac{y^{-k}}{(-k)!}\,
  \ov{\rmR_k^{-k}\, f}
\tx{.}
\end{gather}
Closely related to this, we record that for all functions~$f$ on~$\HS$ and all~$k \in \ZZ$, $\ga \in \SL{2}(\RR)$, we have
\begin{gather}
\label{eq:mirror_covariance}
  y^k\, \ov{\big( f \big|_k \ga \big)}
=
  \big(y^k\, \ov{f} \big) \big|_{-k} \ga
\tx{.}
\end{gather}

\begin{definition}%
\label{def:polyharmonic_maass_forms}
Let $\Gamma \subset \SL{2}(\ZZ)$ be a finite index subgroup and $\rho :\, \Gamma \ra \GL{}(V(\rho))$ a representation on a finite dimensional, complex vector space~$V(\rho)$. We call a smooth function~$f :\, \HS \ra V(\rho)$ a polyharmonic weak Maa\ss{} form for~$\rho$ of weight~$k \in \ZZ$ and depth~$d \in \ZZ_{\ge 0}$, if
\begin{enumerateroman}
\item $f |_k \ga = \rho(\ga) f$ for all~$\ga \in \Ga$,
\item $\Delta^{d+1}_k\,f = 0$,
\item\label{item:growth} $\| ( f |_k \ga) (\tau) \| \ll \exp(a y)$ as~$y \ra \infty$ for some~$a \in \RR$, some norm~$\|\,\cdot\,\|$ on~$V(\rho)$, and all~$\ga \in \SL{2}(\ZZ)$.
\end{enumerateroman}
\end{definition}

The space of such functions will be denoted by~$\rmH_k^{(d)}(\rho)$. If~$\rho$ is the trivial representation, we may instead write~$\rmH_k^{(d)}(\Ga)$. We say that a nonzero~$f \in \rmH_k^{(d)}(\rho)$ has exact depth~$d$ if~$d = 0$ or if~$f \not\in \rmH_k^{(d-1)}(\rho)$.

\begin{remark}
In analogy with the scalar-valued harmonic case, the first and third condition are compatible with products: Given polyharmonic weak Maa\ss{} forms~$f$ and~$g$ of weights~$k_1$ and~$k_2$ for representations~$\rho_1$ and~$\rho_2$, the tensor product yields a function~$f \otimes g$ that satisfies the first condition of Definition~\ref{def:polyharmonic_maass_forms} for weight~$k_1 + k_2$ and the representation~$\rho_1 \otimes \rho_2$, and the third one holds as well.

The second condition in Definition~\ref{def:polyharmonic_maass_forms}, however, is generally not compatible with products, which leads us to our considerations in Section~\ref{ssec:spectral_families_altering_weight}.
\end{remark}

\begin{remark}
In their seminal work \cite{bruinierfunke04} Bruinier and Funke introduced the notion of harmonic weak Maa\ss{} forms for finite index subgroups of~$\SL{2}(\ZZ)$ and its metaplectic cover (in our notation these are the forms of exact depth $d=0$).
\end{remark}

\begin{remark}
Lagarias--Rhoades~\cite{lagariasrhoades} investigated the space of polyharmonic Maa\ss{} forms and showed that the Taylor coefficients of certain Eisenstein series can be used to build bases for these spaces. Their results were refined by Matsusaka \cite{matsusaka} to apply to spaces of polyharmonic weak Maa\ss{} forms (i.e. forms that satisfy a growth condition as in \ref{item:growth} as opposed to a polynomial growth condition that Maa\ss{} forms satisfy). 
\end{remark}

\begin{example}\label{example:poly}
Following Verdier~\cite{verdier} Bringmann and Kudla introduce certain vector-valued harmonic weak Maa\ss{} forms in \cite{BK}. We let $k \in \ZZ_{\ge 0}$ and set $m=-k$. By $\Pol_m$ we denote the space of polynomials of degree at most $m$ in the variable $X$. The group $\mathrm{SL}_2(\RR)$ acts on $\Pol_m$ via
\begin{gather*}
  \big( \rho_m(\gamma) p \big)(X)
=
  (-cX+a)^m p\left(\mfrac{dX-b}{-cX+a}\right)
\tx{,}\ 
  \gamma = \begin{psmatrix} a&b\\c&d\end{psmatrix}
\tx{.}
\end{gather*}
Let $r\in\ZZ$ with $0\leq r\leq m$ and define
\begin{gather*}
  \frake_{r,m-r}(\tau)(X)
=
  \mfrac{(-1)^{m-r}}{r!} y^{r-m} (X-\tau)^r (X-\overline{\tau})^{m-r}
\tx{.}
\end{gather*}
For $\gamma=\begin{psmatrix} a&b\\c&d\end{psmatrix} \in \mathrm{SL}_2(\RR)$ we then have 
\begin{gather*}
  \frake_{r,m-r}(\gamma\tau)
=
  (c\tau+d)^{m-2r} \rho_m(\gamma) \frake_{r,m-r}(\tau)
\tx{,}
\end{gather*}
that is $\frake_{r,m-r}$ has weight $m-2r$. 
The holomorphic function $\frake_{m,0}$ is a harmonic weak Maa\ss{} form of weight $-m$ and type $\rho_m$. 
\end{example}

The functions $\frake_{r,m-r}$ in Example~\ref{example:poly} behave as follows under the lowering and raising operators:
\begin{gather}
\label{eq:BKbasis_maass_operators}
  \rmL_{m-2r}\frake_{r,m-r}
=
  (r+1)(m-r)\, \frake_{r+1,m-r-1}
\tx{,}\quad
  \rmR_{m-2r}\frake_{r,m-r}
=
  \frake_{r-1,m-r+1}
\tx{.}
\end{gather}
In particular,  $\rmL_{-m}\frake_{m,0}=0$ and $\rmR_{m}\frake_{0,m}=0$. For later reference we also note that
\begin{gather}
\label{eq:BKbasis_laplace}
  \Delta_{m-2r}\, \frake_{r,m-r}
=
  -(r+1)(m-r)\, \frake_{r,m-r}
\end{gather}
and
\begin{gather}
\label{eq:BKbasis_complexconj}
  y^{m-2r}\, \ov{\frake_{r,m-r}}
=
  \frac{(-1)^m (m-r)!}{r!}\, \frake_{m-r,r}
\tx{.}
\end{gather}
Finally, we note that the behavior of the functions in Example~\ref{example:poly} under the flipping operator is
\begin{gather}
  \rmF_{-m} \frake_{m,0} = (-1)^m\, \frake_{m,0}
\tx{.}
\end{gather}

\subsection{Automorphic forms}%
\label{ssec:automorphic_forms}

The map $G \slash K \lar \HH, [g] \mapsto g i$ is an isomorphism of real manifolds and $K$ is the stabilizer of the point $i \in \HH$. This allows us to lift polyharmonic weak Maa\ss{} forms to ``weak'' automorphic forms.

Recall the setting of Definition~\ref{def:polyharmonic_maass_forms}. We associate to $f \in \rmH^{(d)}_k(\rho)$ a smooth function
\begin{gather}
\label{eq:def:lift_to_automorphic_form}
  \varphi_f :\,
  G \ra V(\rho),\ g \lmto \big( f \big|_k g \big)(i)
\tx{.}
\end{gather}
We see that $\varphi_f$ satisfies the following properties:
\begin{gather}
\label{eq:lift_to_automorphic_form_properties}
\begin{alignedat}{2}
  \varphi_f(h g) &= \rho(h) \big(\varphi_f(g)\big)
&&
  \tx{\ for all\ } h \in \Gamma \tx{\ and\ } g \in G
\tx{;}\\
  \varphi_f(g k_\theta) &= \exp(ik\theta)\, \varphi_f(g)
&&
  \tx{\ for all\ } g \in G \tx{\ and\ } \theta \in \RR
\tx{,}
\end{alignedat}
\end{gather}
where $k_\theta = \exp(i \theta H) \in K$ is given in~\eqref{E:Rotation}. The second relation in~\eqref{eq:lift_to_automorphic_form_properties} implies that we have $H \varphi_f = k \varphi_f$. Further, the condition~$\Delta_k^{d+1}\,f = 0$ in Definition~\ref{def:polyharmonic_maass_forms} translates to
\begin{gather}
\label{eq:lift_to_automorphic_form_casimir}
  \bigl(C - (k^2-2k)\bigr)^{d + 1}\, \varphi_f = 0
\tx{,}
\end{gather}
where~$C$ is the Casimir element in~\eqref{eq:def:casimir_element}, which generates~$\frakz \subset U(\lieg)$.

Consider the vector space
\begin{gather}
\label{eq:def:automorphic_forms_arbitrary_growth}
  \kA\bigl(G, \Gamma, \rho\bigr)
:= 
  \left\{
  \HH \xlongrightarrow{\varphi} W \;:\;
  \begin{aligned}
  & \varphi(h g) = \rho(h)\bigl(\varphi(g)\bigr)
    \tx{\ for all\ }h \in \Gamma, g \in G
  \tx{,}\\
  & \varphi \tx{\ is $K$-finite and~$\frakz$-finite}
  \end{aligned}
  \right\}
\tx{.}
\end{gather}
The space $\kA\bigl(G, \Gamma, (W, \rho)\bigr)$ is naturally a $(\lieg, K)$--module and for any $f \in \rmH^{(d)}_k(\rho)$ we have: $\varphi_f \in \kA(G, \Gamma, \rho)$.

\begin{remark}
Compared to the definition of the space of automorphic forms the space in~\eqref{eq:def:automorphic_forms_arbitrary_growth} lacks an (exponential) growth condition mirroring the one in Definition~\ref{def:polyharmonic_maass_forms}. Since we are merely interested in $(\lieg, K)$--submodules that are generated by~$\varphi_f$ for $f \in \rmH^{(d)}_k(\rho)$, this does not affect our further discussion.
\end{remark}

\subsection{Harish-Chandra modules and quiver representations}%
\label{ssec:harish_chandra_modules_quiver_representations}

We continue to work in the setting of Definition~\ref{def:polyharmonic_maass_forms}. We attach to a function $f \in \rmH_k(\rho)$ the Harish-Chandra module $M(\varphi_f) \subset \kA(G, \Gamma, \rho)$ generated by the function $\varphi_f$.

Let
\begin{gather*}
  \gamma = k^2 - 2k = (k-1)^2 - 1
\tx{.}
\end{gather*}
We define~$l \in \ZZ_{\ge 0}$ according to the following cases:
\begin{gather}
\label{eq:def:block_parameter_from_weight}
\begin{alignedat}{2}
  l &= 1 -k \tx{,}\quad &&\tx{if\ } k < 1
\tx{;}\\
  l &= 0 = 1 - k = k - 1 \tx{,}\quad &&\tx{if\ } k = 1
\tx{;}\\
  l &= k - 1 \tx{,}\quad &&\tx{if\ } k > 1
\tx{.}
\end{alignedat}
\end{gather}
With this notion, we have~$\gamma = l^2-1$ and~$M(\varphi_f) \in \HC_l(\lieg, K)$. We assume that~$f \ne 0$ has exact depth~$d$. Then~$d+1 \in \ZZ_{> 0}$ is the nilpotency degree of $f$ with respect to~$\Delta_k$.

Our next goal is to characterize the quiver representation $\EE\bigl(M(\varphi_f)\bigr)$.

\begin{theorem}
\label{thm:corrspondence_polyharmonic_maass_to_harish_chandra}
For any nonzero~$f \in \rmH^{(d)}_k(\rho)$ of exact depth~$d$, the corresponding quiver representation $V_f := \EE\bigl(M(\varphi_f)\bigr)$ is cyclic. In particular, the Harish-Chandra module $M(\varphi_f)$ is indecomposable. 
\end{theorem}

\begin{proof}
First note that  
\begin{equation}\label{E:HCmoduleSpan}
M(\varphi_f) = \bigl\langle X^{a_1} Y^{b_1} \dots X^{a_s} Y^{b_s} \varphi_f \, \big| \, s \in \ZZ_{\ge 0}, \ a_1, \dots, a_s, b_1, \dots, b_s \in \ZZ_{\ge 0}\bigr\rangle_{\mathbb{C}}.
\end{equation}
Recall that for any $k \ne 1$ we have an 
equivalence of categories $\HC_l(\lieg, K) \xrightarrow{\EE} \Rep(\dA)$, where $\dA$ is the Gelfand quiver (\ref{E:GelfandQuiver}). 

Consider the case $k < 1$. Then $l = 1-k$. We claim that $V_f$ is a cyclic $\dA$--module of type $\ast$. 
To show this, we use the description of the functor $\EE$ given in Theorem~\ref{T:Equivalence}. Let
\begin{gather*}
  V_f
=
  \left[
  \begin{tikzpicture}[baseline]
  \matrix(m)[matrix of math nodes,
  column sep = 2.5em,
  ampersand replacement=\&]
  { V_- \& V_\ast \& V_+ \\};
  \path[-stealth,bend left=15]
  (m-1-1.east |- m-1-2.north) edge node[above] {$A_-$} (m-1-2.west |- m-1-2.north)
  (m-1-2.west |- m-1-2.south) edge node[below] {$B_-$} (m-1-1.east |- m-1-2.south);
  \path[-stealth,bend right=15]
  (m-1-3.west |- m-1-2.north) edge node[above] {$A_+$} (m-1-2.east |- m-1-2.north)
  (m-1-2.east |- m-1-2.south) edge node[below] {$B_+$} (m-1-3.west |- m-1-2.south);
  \end{tikzpicture}
  \right]
= 
  \left[
  \begin{tikzpicture}[baseline]
  \matrix(m)[matrix of math nodes,
  column sep = 2.5em,
  ampersand replacement=\&]
  { M_{-l-1} \& M_{-l+1} \& M_{l+1} \\};
  \path[-stealth,bend left=15]
  (m-1-1.east |- m-1-1.north) edge node[above] {$\Xm$}          (m-1-2.west |- m-1-1.north)
  (m-1-2.east |- m-1-1.north) edge node[above] {$\Xp \Xz$}      (m-1-3.west |- m-1-1.north)
  (m-1-3.west |- m-1-1.south) edge node[below] {$\Xz^{-1} \Yp$} (m-1-2.east |- m-1-1.south)
  (m-1-2.west |- m-1-1.south) edge node[below] {$\Ym$}          (m-1-1.east |- m-1-1.south);
  \end{tikzpicture}
  \right]
\tx{.}
\end{gather*}
We have the following identifications: $A_{-} = X_{-}$, $B_{-} = Y_{-}$ and $B_{+} = X_{+} X_{\ast}$. 
Recall that $A_- B_-  = C_\ast = A_+ B_+$. It follows from (\ref{E:HCmoduleSpan}), the identities (\ref{E:KeyIdentities}) and the fact that $C$ is central in $U(\lieg)$ that the following statements are true:
\begin{gather}\label{E:SpacesCyclicModule}
  V_\ast
=
  \mathrm{span}_{\CC}\, \bigl\{C_\ast^n \varphi_f \, \big| \,  n \in \ZZ_{\ge 0} \bigr\}
\quad\tx{and}\quad 
  V_\pm
=
  \mathrm{span}_{\CC}\, \bigl\{B_\pm C_\ast^n \varphi_f \, \big| \, n \in \ZZ_{\ge 0} \bigr\}
\tx{.}
\end{gather}
Now let us view $\Rep(\dA)$ as the category of finite dimensional modules over the $\CC$-algebra $\dA$ given by (\ref{E:GelfandOrder}). Let $e_\ast$ be the primitive idempotent corresponding to the vertex $\ast$ of the Gelfand quiver; see (\ref{E:Idempotents}).  Note that for any $W \in \mathrm{Ob}\bigl(\Rep(\dA)\bigr)$,  we have an isomorphism of complex vector spaces 
$$
\Hom_{\dA}(\dA e_\ast, W) \cong e_\ast W = W_\ast,
$$
which assigns to a homomorphism $\dA e_\ast \stackrel{\vartheta}\lar W$  the element $\vartheta(e_\ast) \in W_\ast$. 

Consider now the homomorphism of $\dA$-modules $\dA e_\ast \xlongrightarrow{\xi} V_f$ which is determined by the condition  $\xi(e_\ast) = \varphi_f \in M_{-l+1} = V_\ast$.
It follows from $\dA$-linearity of $\xi$ and the description (\ref{E:SpacesCyclicModule}) of the vector spaces $V_\ast$ and $V_\pm$  that 
$
  \xi
$
is an epimorphism. As a consequence, $V_f$ is cyclic of type $\ast$, as asserted. 
Moreover, the nilpotency degree $d+1$ of $f$ is equal to $m_\ast$ (which is the nilpotency degree of $C_\ast$). Since $\EE$ is an equivalence of categories, the Harish-Chandra module $M(\varphi_f)$ is indecomposable. 

The case  $k > 1$ is analogous. We have: $k = l+1$ and $V_f$ is a cyclic $\dA$--module of type $+$. We have: $\varphi_f \in V_+$ and $d + 1 = m_+$  is the nilpotency degree of the endomorphism $C_+ = B_+ A_+$. 

The case $k = 1$ is exceptional since in this case we have an equivalence of categories $\HC_0(\lieg, K) \xrightarrow{\EE} \Rep(\dB)$, where $\dB$ is given by (\ref{E:CyclicQuiver}). In this case
\begin{gather*}
  V_f
=
  \left[
  \begin{tikzpicture}[baseline]
  \matrix(m)[matrix of math nodes,
  column sep = 4em,
  ampersand replacement=\&]
  { V_- \& V_+ \\};
  \path[-stealth,bend left=15]
  (m-1-1.east |- m-1-1.north) edge node[above] {$Z_+$} (m-1-2.west |- m-1-1.north)
  (m-1-2.west |- m-1-1.south) edge node[below] {$Z_-$} (m-1-1.east |- m-1-1.south);
  \end{tikzpicture}
  \right]
= 
  \left[
  \begin{tikzpicture}[baseline]
  \matrix(m)[matrix of math nodes,
  column sep = 4em,
  ampersand replacement=\&]
  { M_{-1} \& M_1 \\};
  \path[-stealth,bend left=15]
  (m-1-1.east |- m-1-1.north) edge node[above] {$X$} (m-1-2.west |- m-1-1.north)
  (m-1-2.west |- m-1-1.south) edge node[below] {$Y$} (m-1-1.east |- m-1-1.south);
  \end{tikzpicture}
  \right]
\end{gather*}
is a cyclic representation of $\dB$ of type $+$ and $d + 1 = m_+$ is the nilpotency degree of $C_+ := Z_+ Z_-$.
\end{proof}

\begin{remark}
It is not true in general that a $\lieg$-module generated by a cyclic vector is automatically indecomposable.\footnote{The second-named author is grateful to Volodymyr Mazorchuk for drawing his attention to this fact.} Indeed, let 
$U$ and $U'$ be two non-isomorphic finite dimensional simple representations of $\lieg$. Then there are precisely two non-trivial  $\lieg$-submodules of $U \oplus U'$: $\left\{(u, 0)\big| u \in U\right\}$ and $\left\{(0, u')\big| u' \in U'\right\}$ (this statement  follows easily from the fact any finite dimensional $\lieg$-module is a direct sum of simple ones and there are no non-zero morphisms between non-isomorphic simple representations). As a consequence, for any $0 \ne u \in U$ and $0 \ne u' \in U'$ the vector $(u, u') \in U \oplus U'$ is cyclic. 
\end{remark}

\subsection{Classification by representation theoretic labels}%
\label{ssec:classification_representation_labels}

Let $f \in \rmH^{(d)}(\rho)$ be a polyharmonic form of weight $k \in \ZZ$ and exact depth~$d$. Recall the operators~$C_\ast = A_- B_-$ and~$C_+ = Z_+ Z_-$ from the proof of Theorem~\ref{thm:corrspondence_polyharmonic_maass_to_harish_chandra}. The correspondence between the Harish-Chandra module $M(\varphi_f)$ and the cyclic quiver representation $V_f = \EE\bigl(M(\varphi_f)\bigr)$ classified in Theorem~\ref{T:GelfandQuiverCyclicReps} and Remark~\ref{rm:CyclicQuiverCyclicReps} is as follows. Labels prefixed with~G stand for representations of the Gelfand quiver, and those prefixed with~C stand for representations of the cyclic quiver.

\begin{enumerate}
\item[(GI)]
\Needspace*{5\baselineskip}
Let $k < 1$. Then $V_f$ be a cyclic representation of $\dA$ of type ``$\ast$'' and $m_\ast = d+1$. Writing $\psi_f := C_\ast^{d}(\varphi_f)$ we have the following cases: 
\begin{enumerate}[label=(\alph*),itemsep=2pt,parsep=0pt]
\item $V_f$ has type~\hyperref[it:representation_gelfand:Ia]{GIa}
  if $\Ym \psi_f  =  0$ and $\Xp \Xz \psi_f  =  0$.
\item $V_f$ has type~\hyperref[it:representation_gelfand:Ib]{GIb}
  if $\Ym \psi_f \ne 0$ and $\Xp \Xz \psi_f \ne 0$.
\item $V_f$ has type~\hyperref[it:representation_gelfand:Ic]{GIc}
  if $\Ym \psi_f  =  0$ but $\Xp \Xz \psi_f \ne 0$.
\item $V_f$ has type~\hyperref[it:representation_gelfand:Id]{GId}
  if $\Ym \psi_f \ne 0$ and $\Xp \Xz \psi_f  =  0$.
\end{enumerate}

\item[(GII)]
Let $k >  1$. Then $V_f$ be a cyclic representation of $\dA$ of type ``$+$'' and $m_+ = d+1$. 
\begin{enumerate}[label=(\arabic*)]
\item
\Needspace*{5\baselineskip}
Assume that $d = 0$. We have the following cases:
\begin{enumerate}[label=(\alph*),itemsep=2pt,parsep=0pt]
\item $V_f$ has type~\hyperref[it:representation_gelfand:IIa]{GIIa}
  if $\Yp \varphi_f  =  0$.
\item $V_f$ has type~\hyperref[it:representation_gelfand:IIb]{GIIb}
  if $\Yp \varphi_f \ne 0$ but $\Ym \Yz \Yp \varphi_f = 0$.
\item $V_f$ has type~\hyperref[it:representation_gelfand:IIc]{GIIc}
  if $\Ym \Yz \Yp \varphi_f \ne 0$.
\item Type~\hyperref[it:representation_gelfand:IId]{GIId} does not occur for $d = 0$.
\end{enumerate}

\item
\Needspace*{5\baselineskip}
Assume that $d \ge 1$. Writing~$\psi_f := C_\ast^{d-1} \Yp \varphi_f$, where we note that $\Xp \psi_f = C_+^{d} \varphi_f \ne 0$, we have the following cases:
\begin{enumerate}[label=(\alph*),itemsep=2pt,parsep=0pt]
\item $V_f$ has type~\hyperref[it:representation_gelfand:IIa]{GIIa}
  if $\Ym \Yz \psi_f \ne 0$ and $C_\ast \psi_f = 0$.
\item $V_f$ has type~\hyperref[it:representation_gelfand:IIb]{GIIb}
  if $\Ym \Yz \psi_f \ne  0$, $C_\ast \psi_f \ne 0$, and $\Ym \Yz C_\ast \psi_f = 0$.
\item $V_f$ has type~\hyperref[it:representation_gelfand:IIc]{GIIc}
  if $\Ym \Yz \psi_f \ne 0$ and $\Ym \Yz C_\ast \psi_f \ne 0$.
\item $V_f$ has type~\hyperref[it:representation_gelfand:IId]{GIId}
  if $\Ym \Yz \psi_f  =  0$.
\end{enumerate}
\end{enumerate}

\item[(CI)]
\Needspace*{4\baselineskip}
Let $k = 1$.  Then $V_f$ is a cyclic representation of $\dB$ of type ``$+$'' and $d = m_+ + 1$. Writing $\psi_f := C_+^{d} \varphi_f$ we have the following cases:
\begin{enumerate}[label=(\alph*),itemsep=2pt,parsep=0pt]
\item $V_f$ has type~\hyperref[it:representation_cyclic:Ia]{CIa}
  if $Y \psi_f  =  0$.
\item $V_f$ has type~\hyperref[it:representation_cyclic:Ib]{CIb}
  if $Y \psi_f \ne  0$.
\end{enumerate}
\end{enumerate}

\subsection{Classification by BK-style labels}%
\label{ssec:classification_BK_labels}

The classification in Section~\ref{ssec:classification_representation_labels} matches with the one obtained by Bringmann and Kudla in Theorem~5.2 of~\cite{BK} in the case $d = 0$. However, the labels, which are natural from the point of view of representation theory, do not agree with the ones of Bringmann--Kudla, which are natural from the perspective of modular forms. Table~\ref{tab:label_translation} provides a translation between them. In higher depth we encounter an additional Case~IIId, which corresponds to the representation theoretic Case~GIId.

Let $f \in \rmH^{(d)}(\rho)$ be a polyharmonic form of weight $k \in \ZZ$ and exact depth~$d$. The Harish-Chandra module~$M(\varphi_f)$ as in Section~\ref{ssec:classification_representation_labels} is classified using the labels of Bring\-mann--Kudla as follows:
\begin{enumerate}[label=(\Roman*)]
\item
\Needspace*{4\baselineskip}
$k < 1$.
\begin{enumerate}[label=(I\alph*),itemsep=0pt,parsep=0pt]
\item $\rmL\, \Delta^d\, f  =  0$ and~$\rmR^{1-k}\, \Delta^d\, f  =  0$.
\item $\rmL\, \Delta^d\, f  =  0$ and~$\rmR^{1-k}\, \Delta^d\, f \ne 0$.
\item $\rmL\, \Delta^d\, f \ne 0$ and~$\rmR^{1-k}\, \Delta^d\, f  =  0$.
\item $\rmL\, \Delta^d\, f \ne 0$ and~$\rmR^{1-k}\, \Delta^d\, f \ne 0$.
\end{enumerate}

\item
\Needspace*{4\baselineskip}
$k = 1$.
\begin{enumerate}[label=(II\alph*),itemsep=0pt,parsep=0pt]
\item $\rmL\, \Delta^d\,f = 0$.
\item $\rmL\, \Delta^d\,f \ne 0$.
\end{enumerate}

\item
\Needspace*{4\baselineskip}
$k > 1$.
\begin{enumerate}[label=(III\alph*),itemsep=0pt,parsep=0pt]
\item $\rmL^k\, \Delta^{d-1}\,f \ne 0$, if~$d \ge 1$, and $\rmL\, \Delta^d\, f = 0$.
\item $\rmL\, \Delta^d\,f \ne 0$ and~$\rmL^k\, \Delta^d\,f = 0$.
\item $\rmL^k\, \Delta^d\, f \ne 0$.
\item $d \ge 1$ and $\rmL^k\, \Delta^{d-1}\,f = 0$.
\end{enumerate}
\end{enumerate}

To translate the classification in Section~\ref{ssec:classification_representation_labels}, it suffices to recall the notation in~\eqref{eq:def:XYrestrictions_cyclic} and~\eqref{eq:def:XYrestrictions}. The operators~$X$ and~$Y$ restricted to~$M_k$ yield nonzero multiples of~$\rmL_k$ and~$\rmR_k$. In Cases~IIIb and~IIIc, we have removed extraneous conditions. For example, in Case~IIIc $\rmL^k\, \Delta^d\, f \ne 0$ implies $\rmL^k\, \Delta^{d-1}\,f \ne 0$, which is therefore omitted.

\begin{figure}[H]
\label{tab:label_translation}
\caption{Translation between labels assigned by Bringmann--Kudla and representation theoretic labels emerging in Section~\ref{sec:cyclic_modules}, including page references to their modular and representation theoretic realizations.}
\begin{tabular}{lcccccccccc}
\toprule
BK label &
\hyperref[sssec:modular_realization:Ia]  {Ia} &
\hyperref[sssec:modular_realization:Ib]  {Ib} &
\hyperref[sssec:modular_realization:Ic]  {Ic} &
\hyperref[sssec:modular_realization:Id]  {Id} &
\hyperref[sssec:modular_realization:IIa] {IIa} &
\hyperref[sssec:modular_realization:IIb] {IIb} &
\hyperref[sssec:modular_realization:IIIa]{IIIa} &
\hyperref[sssec:modular_realization:IIIb]{IIIb} &
\hyperref[sssec:modular_realization:IIIc]{IIIc} &
\hyperref[sssec:modular_realization:IIId]{IIId}
\\
mod.\@ form on p.\@ &
\pageref{sssec:modular_realization:Ia} &
\pageref{sssec:modular_realization:Ib} &
\pageref{sssec:modular_realization:Ic} &
\pageref{sssec:modular_realization:Id} &
\pageref{sssec:modular_realization:IIa} &
\pageref{sssec:modular_realization:IIb} &
\pageref{sssec:modular_realization:IIIa} &
\pageref{sssec:modular_realization:IIIb} &
\pageref{sssec:modular_realization:IIIc} &
\pageref{sssec:modular_realization:IIId}
\\
\cmidrule(r){2-5}
\cmidrule(r){8-11}
repr.\@ label &
\hyperref[it:representation_gelfand:Ia]{GIa} &
\hyperref[it:representation_gelfand:Ic]{GIc} &
\hyperref[it:representation_gelfand:Id]{GId} &
\hyperref[it:representation_gelfand:Ib]{GIb} &
\hyperref[it:representation_cyclic:Ia]{CIa} &
\hyperref[it:representation_cyclic:Ib]{CIb} &
\hyperref[it:representation_gelfand:IIa]{GIIa} &
\hyperref[it:representation_gelfand:IIb]{GIIb} &
\hyperref[it:representation_gelfand:IIc]{GIIc} &
\hyperref[it:representation_gelfand:IId]{GIId}
\\
repr. on p.\@ &
\pageref{it:representation_gelfand:Ia} &
\pageref{it:representation_gelfand:Ic} &
\pageref{it:representation_gelfand:Id} &
\pageref{it:representation_gelfand:Ib} &
\pageref{it:representation_cyclic:Ia} &
\pageref{it:representation_cyclic:Ib} &
\pageref{it:representation_gelfand:IIa} &
\pageref{it:representation_gelfand:IIb} &
\pageref{it:representation_gelfand:IIc} &
\pageref{it:representation_gelfand:IId}
\\
\bottomrule
\end{tabular}
\end{figure}

\section{Construction of spectral derivatives and modular realizations}%
\label{sec:spectral_taylor_coefficients}

In this section we give an existence theorem for spectral derivatives.  In particular,  we employ our theorem in Section~\ref{sec:modular_realization} to provide examples that realize all possible modules that occur in our classification.

\subsection{Commutation relations for differential operators}

To perform the calculations in Section~\ref{ssec:spectral_families_altering_weight}, we need several algebraic relations for the Maa\ss{} operators. We preserve the subscripts of all operators, but when viewing~$\Delta$, $\rmL$, and~$\rmR$ as graded operators they can be suppressed for clarity. In particular, the next relations can be verified by merely using the commutator~$[\rmL,\rmR] = -k$, which follows from~\eqref{eq:laplace} and in which~$k$ on the right hand side is viewed as a graded scalar as well.
We have the commutator relations for the Laplace operator:
\begin{gather}
\label{eq:laplace_commutators}	
\begin{alignedat}{4}
&
  \Delta_{k-2r}\, &&\rmL_k^r
&&=
  \rmL_k^r\, &&\big( \Delta_k - r(k - r - 1) \big)
\tx{,}\\
&
  \Delta_{k+2r}\, &&\rmR_k^r
&&=
  \rmR_k^r\, &&\big( \Delta_k + r(k + r - 1) \big)
\tx{;}
\end{alignedat}
\end{gather}
And for the Maa\ss{} lowering and raising operators:
\begin{gather}
\label{eq:maass_iterated_commutators}
\begin{alignedat}{4}
&
  \rmR_{k-2r} && \rmL_k^r
=
  - \rmL_k^{r-1} && \big( \Delta_k - (r-1)(k-r)) \big)
\tx{,}\\
&
  \rmL_{k+2r} && \rmR_k^r
=
  - \rmR_k^{r-1} && \big( \Delta_k + r(k + r-1) \big)
\tx{.}
\end{alignedat}
\end{gather}

The Bol Identity asserts that for an integer~$k \le 0$ and a weight-$k$ harmonic function~$f$ we have
\begin{gather}
\label{eqref:bol_identity}
  \rmR_k^{1-k}\, f
=
  (2 \pi \partial_\tau)^{1-k}\,f
\tx{.}
\end{gather}
A straightforward calculation shows that for smooth functions~$f$ and~$g$ and integers~$k_1$, $k_2$ we have
\begin{gather}
\label{eq:deltaonproducts}
  \Delta_{k_1+k_2} (f \cdot g)
=
  (\Delta_{k_1} f) \cdot g
\,+\,
  f \cdot (\Delta_{k_2}g)
\,-\,
  (\rmR_{k_1} f) \cdot (\rmL_{k_2} g)
\,-\,
  (\rmL_{k_1} f)\cdot (\rmR_{k_2} g)
\tx{.}
\end{gather}

\begin{proposition}
\label{prop:flipping_commutators}
The intertwining relations for the flipping operator and the Maa\ss{} operators, $k \le 0$, are
\begin{gather}
\label{eq:flipping_commutators}
\begin{aligned}
  \rmL_k\, \rmF_k \Delta_k
&{}=
  -(k-2)(k-1) \rmF_{k-2}\, \rmL_k
\tx{,}
\\
  - k (k+1)\,
  \rmR_k\, \rmF_k
&{}=
  \rmF_{k+2}\, \rmR_k (\Delta_k + k)
\tx{.}
\end{aligned}
\end{gather}
\end{proposition}
\begin{proof}
We begin with the first identity and evaluate by a short calculation
\begin{gather*}
  \rmL_k\, \mfrac{y^{-k}}{(-k)!} \ov{\rmR_k^{-k}\Delta_k\, f}
=
  \mfrac{y^{2-k}}{(-k)!}\,
  \ov{\rmR_k^{-k+1} \Delta_k\, f}
\tx{.}
\end{gather*}
Now we write $\Delta_k = - \rmR_{k-2}\rmL_k$ and see
\begin{gather*}
  \mfrac{y^{-k+2}}{(-k)!}\,
  \ov{\rmR_k^{-k+1} \Delta_k\, f}
= 
  \mfrac{-y^{-k+2}}{(-k)!}\,
  \ov{\rmR_k^{-k+1} \rmR_{k-2} \rmL_k\, f}
=
  -(-k+2)(-k+1)\, \rmF_{k-2} \rmL_k\, f
\tx{.}
\end{gather*}

For the second identity we calculate
\begin{align*}
  \rmR_k \rmF_k\, f
=
  \mfrac{y^{-k-2}}{(-k)!}\,
  \ov{\rmL_{-k} \rmR_k^{-k}\, f}
  \tx{.}
  \end{align*}
Then we use \eqref{eq:laplace_commutators} and obtain
\begin{align*}
  \mfrac{y^{-k-2}}{(-k)!}\,
  \ov{\rmR_k^{-k-1} \big(\Delta_k - k (k - k - 1)\big)\, f}
=
  \mfrac{-1}{(-k)(-k-1)}\,
  \rmF_{k+2} \rmR_k \big(\Delta_k + k\big)\, f
\tx{.}
\end{align*}
\end{proof}

\begin{proposition}%
\label{prop:flipping_laplace_and_involution}
We have
\begin{gather}
\label{eq:flipping_laplace_and_involution}
\begin{aligned}
  \Delta_k\, \rmF_k
&{}=
  \rmF_k\, \Delta_k
\tx{,}\\
  \rmF_k\, \rmF_k
&{}=
  \mfrac{(-1)^{-k}}{(-k)!^2}\,
  \big(\Delta_k + 1k\big)
  \big(\Delta_k + 2(k+1)\big)
  \cdots
  \big(\Delta_k + (-k)(-1)\big)\,
\tx{.}
\end{aligned}
\end{gather}
\end{proposition}
\begin{proof}
For the first identity we write~$f = \Delta_k g$ locally and use both identities in~\eqref{eq:flipping_commutators} as well as~\eqref{eq:laplace_commutators} to obtain
\begin{align*}
  \Delta_k \rmF_k\, f
&=
  - \rmR_{k-2} \rmL_k \rmF_k \Delta_k\, g
=
  \rmR_{k-2}
  (k-2)(k-1)\,
  \rmF_{k-2} \rmL_k\, g
\\
&=
  \mfrac{-(k-2)(k-1)}{(k-2)(k-1)}
  \rmF_k \rmR_{k-2} (\Delta_{k-2} + k - 2)\, \rmL_k g
\\
&=
  -
  \rmF_k (\Delta_k - k + 2 + k - 2)\, \rmR_{k-2} \rmL_k\, g
=
  \rmF_k \Delta_k^2\, g
=
  \rmF_k \Delta_k\, f
\tx{.}
\end{align*}
For the second identity we use that for any real-analytic function $g$ we have the identity
\begin{gather*}
  \rmR_k \big( y^{-k}\, \ov{\rmR_{-k-2}g}\big)
=
  y^{-k-2}\,  \ov{(\Delta_{-k-2} +k+2)g}
\tx{.}
\end{gather*}
A repeated application of this identity yields the result.
\end{proof}

\begin{remark}
To maintain the algebraic perspective of Section~\ref{sec:polyharmonic_to_quiver} on polyharmonic Maa\ss{} forms, we can lift the flipping operator to the automorphic forms as remarked on page~1738 of~\cite{BK}:
\begin{gather*}
  \phi_{\rmF_k\, f}
=
  \frac{1}{(-k)!} \ov{ X^{-k} \phi_f }
\tx{.}
\end{gather*}
Then Propositions~\ref{prop:flipping_commutators} and~\ref{prop:flipping_laplace_and_involution} can be verified by a Lie algebra computation. This drastically reduces the calculation required for~\eqref{eq:flipping_commutators} and the first relation in~\eqref{eq:flipping_laplace_and_involution}.
\end{remark}

A direct calculation shows that for smooth functions~$f$ we have 
\begin{gather}
\label{eq:flipping_maass_low}
  \rmL_k\, \rmF_k\, f
=
  \mfrac{y^{2-k}}{(-k)!}\, \ov{\rmR^{1-k}_k\,f}
\tx{.}
\end{gather}

As a consequence of Proposition~\ref{prop:flipping_laplace_and_involution} if~$\Delta_k\, f = 0$, we have~$\rmF_k \rmF_k\, f = f$ and, compare for example with Proposition 5.14 of \cite{bringmann-folsom-ono-rolen-2018},
\begin{gather}
\label{eq:flipping_maass_bol}
  \rmR_k^{1-k}\, \rmF_k\, f
=
  (-k)! y^{k-2}\, \ov{ \rmL_k\, f}
\tx{.}
\end{gather}

\subsection{Eisenstein and Poincar\'e series}

In this section we revisit some Eisenstein and Poincar\'e series that we will later use as an input for spectral families when constructing examples of polyharmonic Maa\ss{} forms. Let~$\Gamma_\infty \subset \SL{2}(\ZZ)$ be the subgroup of upper triangular matrices.

We first define the Eisenstein series of weight~$k \in \ZZ$. Let~$s\in\CC$ with~$\Re(s)>1-\frac{k}{2}$. Then we set
\begin{gather}\label{eq:def:eisensteink}
  E_k(\tau,s)
=
  \sum_{\gamma\in \Gamma_\infty \backslash \SL{2}(\ZZ)} y^s\big|_k\, \gamma
\tx{.}
\end{gather}
Via analytic continuation~$E_k(\tau,s)$ extends to all~$s \in \CC$ except for possible simple poles. We set~$E_k(\tau) = E_k(\tau,0)$. The Maa\ss{} lowering and raising operators act as
\begin{gather}%
\label{eq:eis:maass_operators}
\begin{aligned}
  \rmL_k\, E_k(\tau, s)
&=
  s\,
  E_{k-2}(\tau,s+1)
\tx{,}
\\
  \rmR_k\, E_k(\tau, s)
&=
  (s+k)\,
  E_{k+2}(\tau,s-1)
\tx{.}
\end{aligned}
\end{gather}
In particular, we have~$\Delta_k\, E_k(\tau,s) = s(1-k-s) E_k(\tau,s)$. Moreover, we have
\begin{gather}
  y^k\, \ov{E_k(\,\cdot\,, s)} 
=
  E_{-k}(\,\cdot\,, \ov{s}+k)
\tx{.}
\end{gather}

Now we recall some facts on Poincar\'{e} series with exponential growth at the cusps following~\cite{bruinier-2002a}. We let~$M_{\nu,\mu}(z)$ and~$W_{\nu,\mu}(z)$ denote the usual Whittaker functions (see p.~190 of~\cite{pocket}). For integers~$k$ and~$m \ne 0$, $\tau=x+iy\in\HS$, and $s\in \CC$ with $\Re(s)>1$, we define
\begin{multline}%
\label{eq:def:poincare}
  F_{k,m}(z,s) 
=
\\
  \mfrac{1}{2\Gamma(2s)}
  \sum_{\gamma \in \Gamma_\infty \backslash \SL{2}(\ZZ)} \Big(
  \big(-\sgn(m)\big)^{1-k}
  \big(4\pi |m| y\big)^{-\frac{k}{2}}
  M_{\sgn(m)\frac{k}{2}, s-\frac12}\big(4\pi |m| y\big)\,
  e(mx)
  \Big) \Big|_k\, \gamma
\tx{.}
\end{multline}
This Poincar\'{e} series converges for~$\Re(s)>1$ and is an eigenfunction of $\Delta_k$ with eigenvalue $s(1-s)+(k^2-2k)\slash 4$. 

We recall its behavior under the Maa\ss{} raising and lowering operators.

\begin{proposition}
We have
\begin{gather}
\begin{aligned}
 \rmR_k\, F_{k,m}(\tau,s) 
 & = 
 4\pi |m|\, \big( s+ \tfrac{k}{2}\big)\,  F_{k+2,m}(\tau,s) 
 \tx{,}
 \\
 \rmL_k  \, F_{k,m}(\tau,s) 
 & = 
 \mfrac{1}{4\pi |m|}\, \big( s-\tfrac{k}{2}\big)\,  F_{k-2,m}(\tau,s) 
  \tx{.}
\end{aligned}
\end{gather}
\end{proposition}
\begin{proof}
Since $\rmL_k$ and $\rmR_k$ commute with the slash operator, it suffices to show the identity on the corresponding Whittaker functions. We use equations~(13.4.10),~(13.4.11) and~(13.1.32) in \cite{pocket} which imply the desired identity. Note that parts of these identities were already proven in~\cite{bruinier-ono-partition} and~\cite{alfes2014}. 
\end{proof}

\begin{proposition}
Complex conjugation yields
\begin{gather}
\label{eq:pcseries:conj}
  y^k\, \ov{F_{k,m}(\tau,s)}
=
  (-1)^{1-k}
  (4 \pi |m|)^{-k}\,
  F_{-k,-m}(\tau,\ov{s})
\tx{.}
\end{gather}
\end{proposition}
\begin{proof}
We have 
\begin{align*}
&
 y^k \Big(\ov{
  \mfrac{1}{2\Gamma(2s)}
  \big(-\sgn(m)\big)^{1-k}
  \big(4\pi |m| y\big)^{-\frac{k}{2}}
  M_{\sgn(m)\frac{k}{2}, s-\frac12}\big(4\pi |m| y\big)\,
  e(mx)
  \Big)\big|_{k}\gamma
  }
  \\
={}& 
  y^k
  \mfrac{1}{2\Gamma(2\ov{s})}
  \big(-\sgn(m)\big)^{1-k}
  \big(4\pi |m| \Im(\gamma \tau)\big)^{-\frac{k}{2}}
  \\
  & \quad \quad \times 
  M_{\sgn(m)\frac{k}{2}, \ov{s}-\frac12}\big(4\pi |m| \Im(\gamma\tau)\big)\,
  e(-m\Re(\gamma\tau)) (c\ov{\tau}+d)^{-k}.
\end{align*}
Note that $\Im(\gamma\tau)=y (c\tau+d)^{-1} (c\ov{\tau}+d)^{-1}$. A short calculation then yields the desired result.
\end{proof}

We also recall the existence of spectral families of Poincar\'{e} series (compare \cite{bruinier-funke-imamoglu-2015} and \cite{strathausen}).
\begin{proposition}\label{prop:poincarespectral}
 Let $f$ be a harmonic weak Maa\ss{} form of weight $k\leq 0$. There exists an open neighborhood $U$ in $\CC$ of\/~$1-k/2$ and a holomorphic family of functions $(f_s)_{s\in U}$ on $\HS$, where $f_s$ is a weak Maa\ss{} form of weight $k$ of eigenvalue $s(1-s)+(k^2-2k) \slash 4$, and $f_{1-k/2}=f$.
\end{proposition}

\subsection{Derivatives of spectral families: constant weight}
\label{ssec:spectral_families_constant_weight}

Differentials of spectral families are one tool to provide higher depth Maa\ss{} forms. We briefly revisit the case of fixed weight for comparison with the later approach that allows us to change the weight. Prototypical~$f_s$ that fit into the following setup are provided by Eisenstein series and Poincar\'e series.

\begin{lemma}
\label{la:spectral_derivative_laplace_classical}
Let~$k$ be an integer, and~$U \subseteq \CC$ be an open neighborhood of~$0$. Given functions~$f_s :\, \HS \ra V$ for a complex vector space~$V$ that are smooth in~$s$ and~$\tau$, we assume that for all~$\tau \in \HS$ and~$s \in U$ we have
\begin{gather*}
  \Delta_k\,f_s
=
  s (1-k-s) f_s
\tx{.}
\end{gather*}
Then with
\begin{gather*}
  f^{(d)}
\;:=\;
  \big( \partial_s^d\, f_s \big)_{s=0}
\end{gather*}
we have	
\begin{gather}
\label{eq:spectral_derivative_laplace_classical}
  \Delta_k\, f^{(d)}
=
  d
  (1-k)\,
  f^{(d-1)}
\,-\,
  d (d-1)\,
  f^{(d-2)}
\tx{.}
\end{gather}
\end{lemma}
\begin{proof}
Since~$f_s(\tau)$ is smooth in~$s$ and~$\tau$, then~$\partial_s$ and~$\Delta_k$, which is a differential operator with respect to~$\tau$, intertwine. For simplicity, we set~$\partial_s^d f_s = 0$ if~$d < 0$. By the product rule, we have
\begin{align*}
  \Delta_k\, \partial_s^d f_s
&{}=
  \partial_s^d\, \Delta_k f_s
=
  \partial_s^d\, s (1-k-s) f_s
\\
&{}=
  s (1-k-s)\,
  \partial_s^d f_s
\,+\,
  d
  (1-k-2s)\,
  \partial_s^{d-1} f_s
\,-\,
  d (d-1)
  \,
  \partial_s^{d-2} f_s
\tx{.}
\end{align*}
Inserting~$s = 0$ yields the statement.	
\end{proof}

We can employ this lemma to produce polyharmonic Maa\ss{} forms from spectral families. More specifically, we obtain a preimage of~$f^{(0)}$ under~$\Delta_k^d$. Note that the case~$k = 1$ requires special treatment.

\begin{corollary}
\label{cor:spectral_derivative_laplace_classical}
Let~$k$ and~$f^{(d)}$ be as in Lemma~\ref{la:spectral_derivative_laplace_classical}. Then
\begin{gather*}
  \Delta_k^d\,
  \mfrac{1}{d! (1-k)^d}
  f^{(d)}
{}=
  f^{(0)}
\tx{,}
\quad
  \tx{if\ }k \ne 1
\tx{;}
\qquad
  \Delta_k^d\,
  \mfrac{(-1)^d}{(2d)!}
  f^{(2d)}
{}=
  f^{(0)}
\tx{,}
\quad
  \tx{if\ }k = 1
\tx{.}
\end{gather*}
\end{corollary}

\subsection{Altering weights of harmonic Maa\ss{} forms}

The construction of polyharmonic Maa\ss{} forms in Section~\ref{ssec:spectral_families_constant_weight} preserves the weight~$k$ of~$f_s$. We next derive a formalism that allows us to produce polyharmonic Maa\ss{} forms of weight~$k \pm m$ for~$m \in \ZZ_{\ge 0}$ by extending the approach of Section~\ref{ssec:spectral_families_constant_weight}.

\begin{lemma}
\label{la:laplace_BKbasis_eigenform}
Let $k$ and~$m \ge 0$ be integers, and consider a smooth function~$f :\, \HS \ra V$ for a complex vector space~$V$. We assume that~$\Delta_k\,f = \alpha f$ for some~$\alpha \in \CC$.

Writing for integers~$0 \le r \le m$
\begin{gather*}
  f_{\rmL,r}
\;:=\;
  \frake_{r,m-r}\, \rmL^{m-r}_k f
\quad\tx{and}\quad
  f_{\rmR,r}
\;:=\;
  \frake_{r,m-r}\, \rmR^r_k f
\end{gather*}
and~$f_{\rmL,r} = f_{\rmR,r} = 0$ for~$r < 0$ and~$r > m$, we have
\begin{align*}
  \Delta_{k-m}\, f_{\rmL,r}
&{}\;=\;
  \big( \alpha + (m-r)(m-2r-k) \big)\,
  f_{\rmL,r}
\\&\qquad
-\,
  f_{\rmL,r-1}
\,+\,
  (r+1)(m-r) \big( \alpha + (m-r-1)(m-r-k) \big)\,
  f_{\rmL,r+1}
\tx{,}
\\[.3\baselineskip]
  \Delta_{k+m}\, f_{\rmR,r}
&{}\;=\;
  \big( \alpha - r(m-2r-k) - m \big)\,
  f_{\rmR,r}
\\&\qquad
+\,
  \big( \alpha + r (r-1+k) \big)\,
  f_{\rmR,r-1}
\,-\,
  (r+1) (m-r)\,
  f_{\rmR,r+1}
\tx{.}
\end{align*}
\end{lemma}
\begin{proof}
We will apply~\eqref{eq:deltaonproducts} and simplify the contributions arising from the first two terms. To this end, recall from~\eqref{eq:BKbasis_laplace} the Laplace eigenvalues~$-(r+1)(m-r)$ of~$\frake_{r,m-r}$. From~\eqref{eq:laplace_commutators}, we find that
\begin{align*}
  \Delta_{k+2m-2r}\, \rmL_k^{m-r} f
&=
  \big( \alpha + (m-r)(m-r-k+1) \big)\, \rmL_k^{m-r} f
\tx{,}
\\
  \Delta_{k+2r}\, \rmR_k^r f
&=
  \big( \alpha + r(r-1+k) \big)\, \rmR_k^r f
\tx{.}
\end{align*}
The images of~$\frake_{r,m-r}$ under the Maa\ss{} operators are given in~\eqref{eq:BKbasis_maass_operators}. Finally, we obtain from~\eqref{eq:maass_iterated_commutators} that
\begin{align*}
  -
  \rmR_{k-2m+2r}\, \rmL_k^{m-r} f
&=
  \big( \alpha - (m-r-1)(m-r-k) \big)\,
  \rmL_k^{m-r-1} f
\tx{,}
\\
  -
  \rmL_{k-2r}\, \rmR_k^r f
&=
  \big( \alpha - r (r-1+k) \big)\,
  \rmR_k^{r-1} f
\tx{.}
\end{align*}
\end{proof}

\begin{proposition}%
\label{prop:maass_form_change_weight_L}
Given a harmonic weak Maa\ss{} form~$f$ of weight~$k$ and a non-ne\-ga\-tive integer~$m \ge k$, set
\begin{gather*}
  f_\rmL
=
  \sum_{r=0}^{\min\{m,m-k\}}
  \mfrac{1}{(m-r)! (m-r-k)!}
  \frake_{r,m-r} \rmL^{m-r} f
\tx{.}
\end{gather*}
If~$m < k$, then set
\begin{gather*}
  f_\rmL
=
  \sum_{r=0}^m
  \mfrac{1}{(m-r)! (1-k)_{m-r}}
  \frake_{r,m-r} \rmL^{m-r} f
\tx{,}
\end{gather*}
where~$(a)_r = a (a + 1) \cdots (a + r - 1)$ is the Pochhammer symbol.

We have~$\Delta_{k-m}\, f_\rmL = 0$. Further, if~$f_\rmL \ne 0$ then any linear combination of the functions~$\frake_{r,m-r} \rmL^{m-r} f$ that vanishes under~$\Delta_{k-m}$ is a scalar multiple of~$f_\rmL$.
\end{proposition}

\begin{proof}
The case~$m = 0$ is vacuous, and we thus can and will assume that~$m$ is positive. To shorten notation, we adopt the notation~$f_{\rmL,r}$ from Lemma~\ref{la:laplace_BKbasis_eigenform}, write~$c_r$ for the coefficient of~$f_{\rmL,r}$ in the definition of~$f_\rmL$, and write~$c'_r$ for the coefficient of~$f_{\rmL,r}$ in~$\Delta_{k-m}\, f_\rmL$. We will use repeatedly that for~$r < \min\{m, m-k\}$ if~$m \ge k$ and for~$0 \le r < m$ if~$m < k$ we have
\begin{gather*}
  c_{r+1}
=
  (m-r) (m-r-k)\, c_r
\tx{.}
\end{gather*}

We apply Lemma~\ref{la:laplace_BKbasis_eigenform} with~$\alpha = 0$ and the recursion equation for~$c_r$ to find that for~$0 \le r < \min\{m, m-k\}$ if~$m \ge k$ and for~$0 \le r < m$ if~$m < k$ we have
\begin{align*}
  c'_r
={}&
  (m-r)(m-2r-k) c_r - c_{r+1}
\\
&\qquad
  +
  (r-1+1)(m-r+1)(m-r+1-1)(m-r+1-k) c_{r-1}
\\
={}&
  \big(
    \mfrac{m-2r-k}{m-r-k}
  - 1
  + \mfrac{r}{m-r-k}
  \big)\,
  c_{r+1}
=
  0
\tx{.}
\end{align*}

We need a case distinction to check the remaining coefficients~$c'_r$. If~$m = k$, we have
\begin{gather*}
  c'_0 = (m-0)(m-0-k) c_0 = 0
\quad\tx{and}\quad
  c'_1 = (0+1)(m-0)(m-0-1)(m-0-k) c_0 = 0
\tx{.}
\end{gather*}
If~$m \ne k$, we have
\begin{gather*}
  c'_0
=
  (m-0)(m-0-k) c_0 - c_1
=
  (1-1) c_1 = 0
\end{gather*}
by the recursion for~$c_r$. Still assuming that~$m \ne k$, we have for~$r = \min\{m,m-k\}$  if~$m \ge k$ and~$r = m$ if~$m < k$ that
\begin{align*}
  c'_r
&{}=
  (m-r)(m-2r-k) c_r
  +
  (r-1+1)(m-r+1)(m-r+1-1)(m-r+1-k) c_{r-1}
\\
&{}=
  \big(
    (m-r)(m-2r-k)
  + r (m-r)
  \big)\,c_r
=
  (m-r-k) c_r
=
  0
\tx{.}
\end{align*}
Finally, we consider~$m > k > 0$, in which case we have for~$r = m-k+1$
\begin{align*}
  c'_r
=
  (r-1+1)(m-r+1)(m-r+1-1)(m-r+1-k)
  c_{r-1}
=
  0
\tx{.}
\end{align*}
This shows that~$\Delta_{k-m}\, f_\rmL$ vanishes.

To prove the second part of the proposition, we consider the vector space
\begin{gather*}
  \cF
=
  \lspan \CC\big\{
  \frake_{r,m-r} \rmL^{m-r} f
  \,:\,
  0 \le r \le m
  \big\}
\tx{,}
\end{gather*}
which by Lemma~\ref{la:laplace_BKbasis_eigenform} with~$\alpha = 0$ carries an action of~$\Delta_{k-m}$. Also from Lemma~\ref{la:laplace_BKbasis_eigenform} and the contribution of~$\Delta_{k-m} f_{\rmL,r}$ to~$f_{\rmL,r-1}$ stated there, we see that~$\Delta_{k-m}$ yields a surjective map from~$\cF$ to $\cF \slash \CC \frake_{m,0} f$. In particular, its kernel has dimension at most~$1$ and is thus spanned by~$f_\rmL$ if~$f_\rmL \ne 0$.
\end{proof}

\begin{example}
The vector-valued modular form~$\frake_{m,0}$ of weight~$-m$ from Example~\ref{example:poly} that appears in case I(a) of Bring\-mann--Kudla's classification matches the case~$k = 0$ for the constant modular form~$f = 1$ of Proposition~\ref{prop:maass_form_change_weight_R}. Note that~$\rmL^{m-r} 1 = 0$ for $0 \le r < m$.
\end{example}

\begin{proposition}%
\label{prop:maass_form_change_weight_R}
Given a harmonic weak Maa\ss{} form~$f$ of weight~$k$ and an integer~$m > -k$, set
\begin{gather*}
  f_\rmR
=
  \sum_{r=\max\{0,1-k\}}^m
  \mfrac{1}{(m-r)! (r+k-1)!}\,
  \frake_{r,m-r} \rmR_k^r f
\tx{.}
\end{gather*}
If~$m \le -k$, set
\begin{gather*}
  f_\rmR
=
  \sum_{r=0}^m
  \mfrac{1}{(m-r)! (k)_r}\,
  \frake_{r,m-r} \rmR_k^r f
\tx{,}
\end{gather*}
where~$(a)_r$ is the Pochhammer symbol as in Proposition~\ref{prop:maass_form_change_weight_L}.

We have~$\Delta_{k+m}\, f_\rmR = 0$. Further, if~$f_\rmR \ne 0$ then any linear combination of the~$\frake_{r,m-r} \rmR_k^r f$ that vanishes under~$\Delta_{k+m}$ is a scalar multiple of~$f_\rmR$.
\end{proposition}

\begin{proof}
The proof is analogous to the one of Proposition~\ref{prop:maass_form_change_weight_L}. We write~$c_r$ for the coefficient of~$f_{\rmR,r}$ in the definition of~$f_\rmR$, and write~$c'_r$ for the coefficient of~$f_{\rmR,r}$ in~$\Delta_{k+m}\, f_\rmR$, and obtain the relation
\begin{gather*}
  (r+k) c_{r+1} = (m-r) c_r
\end{gather*}
for the nonzero coefficient of~$\frake_{r,m-r} \rmR_k^r f$ in~$f_\rmR$.

Lemma~\ref{la:spectral_derivative_laplace_classical} with~$\alpha = 0$ yields that for~$0 < r < \min\{m, 1-k\}$ if~$m > -k$ and for~$0 < r < m$ if~$m \le -k$ we have
\begin{align*}
  c'_r
={}&
    \big( - r (m-2r-k) - m \big) c_r
  + (r+1)(r+k) c_{r+1}
  - r (m-r+1) c_{r-1}
\\
={}&
    \big( - r (m-2r-k) - m \big) c_r
  + (r+1) (m-r) c_r
  - r (r-1+k) c_r
=
  0
\tx{.}
\end{align*}
The special cases for~$r = 0$, $r = 1-k$, $r = -k$, and~$r = m$ follow the same pattern.

To see that~$\Delta_{k+m}$ is a linear transformation of rank at least~$m$ on
\begin{gather*}
  \cF
=
  \lspan \CC\big\{
  \frake_{r,m-r} \rmR^r f
  \,:\,
  0 \le r \le m
  \big\}
\tx{,}
\end{gather*}
we note that it is surjective onto~$\cF \slash \CC \frake_{0,m} f$ by inspection of the contribution of~$f_{\rmR,r}$ to~$f_{\rmR,r+1}$ in Lemma~\ref{la:laplace_BKbasis_eigenform}.
\end{proof}

\begin{example}
The vector-valued Eisenstein series of weight~$2+m$ in case III(b) of Bringmann--Kudla's classification matches the case~$k = 2$ of Proposition~\ref{prop:maass_form_change_weight_R}. In~(6.10) of~\cite{BK} they consider
\begin{gather}
\label{eq:example_laplace_BKIIIb}
  \sum_{r = 0}^m \mfrac{m!}{(r+1)!(m-r)!}\, \frake_{r,m-r} \rmR^r E_2
\tx{,}
\end{gather}
where~$E_2$ is the modular Eisenstein series of weight~$2$ and level~$1$.
\end{example}

\subsection{Derivatives of spectral families: altering weight}
\label{ssec:spectral_families_altering_weight}

In preparation to the construction of polyharmonic Maa\ss{} forms, we consider the action of the Laplace operator on spectral families.

\begin{lemma}%
\label{la:laplace_BKbasis_spectral_family}
Let~$k$ and~$m \ge 0$ be integers, and~$U \subseteq \CC$ be an open neighborhood of\/~$0$.  Given functions~$f_s :\, \HS \ra V$ for a complex vector space~$V$ that are smooth in~$s$ and~$\tau$, we assume that for all~$\tau \in \HS$ and~$s \in U$ we have
\begin{gather*}
  \Delta_k\,f_s
=
  s (1-k-s) f_s
\tx{.}
\end{gather*}
We set
\begin{gather*}
  f_{\rmL,r}^{(d)}
\;:=\;
  \Big( \partial_s^d\, \frake_{r,m-r}\, \rmL^{m-r}_k f_s \Big)_{s=0}
\quad\tx{and}\quad
  f_{\rmR,r}^{(d)}
\;:=\;
  \Big( \partial_s^d\, \frake_{r,m-r}\, \rmR^r_k f_s \Big)_{s=0}
\tx{,}
\end{gather*}
and if~$d < 0$, $r < 0$, or~$r > m$ set $f_{\rmL,r}^{(d)} = f_{\rmR,r}^{(d)} = 0$. Then we have
\begin{align*}
  \Delta_{k-m}\, f_{\rmL,r}^{(d)}
&{}\;=\;
  (m-r)(m-2r-k)\,
  f_{\rmL,r}^{(d)}
\\&\qquad
-\,
  f_{\rmL,r-1}^{(d)}
\,+\,
  (r+1)(m-r) (m-r-1)(m-r-k)\,
  f_{\rmL,r+1}^{(d)}
\\&\qquad
+\,
  d (1-k)\,
  f_{\rmL,r}^{(d-1)}
\,+\,
  d (r+1)(m-r) (1-k)\,
  f_{\rmL,r+1}^{(d-1)}
\\&\qquad
-\,
  d (d-1)\,
  f_{\rmL,r}^{(d-2)}
\,-\,
  d (d-1) (r+1)(m-r)\,
  f_{\rmL,r+1}^{(d-2)}
\tx{,}
\\[.3\baselineskip]
  \Delta_{k+m}\, f_{\rmR,r}^{(d)}
&{}\;=\;
-\,
  \big( r(m-2r-k) + m \big)\,
  f_{\rmR,r}^{(d)}
\\&\qquad
+\,
  r (r-1+k)\,
  f_{\rmR,r-1}^{(d)}
\,-\,
  (r+1) (m-r)\,
  f_{\rmR,r+1}^{(d)}
\\&\qquad
+\,
  d (1-k)\,
  f_{\rmR,r}^{(d-1)}
\,+\,
  d (1-k)\,
  f_{\rmR,r-1}^{(d-1)}
\\&\qquad
-\,
  d (d-1)\,
  f_{\rmR,r}^{(d-2)}
\,-\,
  d (d-1)\,
  f_{\rmR,r-1}^{(d-2)}
\tx{.}
\end{align*}
\end{lemma}

\begin{proof}
Since~$f_s(\tau)$ is smooth in~$s$ and~$\tau$, we can intertwine differentials with respect to~$s$ and~$\tau$. In particular, we can apply Lemma~\ref{la:laplace_BKbasis_eigenform} to~$f_s$ with~$\alpha = s (1-k-s)$ to compute
\begin{align*}
  \Delta_{k-m} \big(
  \partial_s^d\, \frake_{r,m-r}\, \rmL^{m-r}_k f_s
  \big)
&=
  \partial_s^d \big(
  \Delta_{k-m}\,
  \frake_{r,m-r}\, \rmL^{m-r}_k f_s
  \big)
\quad\tx{and}
\\
  \Delta_{k+m} \big(
  \partial_s^d\, \frake_{r,m-r}\, \rmR^r_k f_s
  \big)
&=
  \partial_s^d \big(
  \Delta_{k+m}\, \frake_{r,m-r}\, \rmR^r_k f_s
  \big)\tx{.}
\end{align*}
The result follows after simplifying the resulting right hand side in Lemma~\ref{la:laplace_BKbasis_eigenform} and setting~$s=0$.
\end{proof}

We are now ready to produce polyharmonic Maa\ss{} forms from spectral families.

\begin{theorem}
\label{thm:spectral_families_altering_weight}
Let~$U \subseteq \CC$ be an open neighborhood of\/~$0$, and $k$ and~$m \ge 0$ be integers. Given functions~$f_s :\, \HS \ra V$ for a complex vector space~$V$ that are smooth in~$s$ and~$\tau$, we assume that for all~$\tau \in \HS$ and~$s \in U$, we have
\begin{gather*}
  \Delta_k\,f_s
=
  s (1-k-s) f_s
\tx{.}
\end{gather*}
For integers~$0 \le r \le m$ and~$d \ge 0$, we write
\begin{gather*}
  f_{\rmL,r}^{(d)}
\;:=\;
  \Big( \partial_s^d\, \frake_{r,m-r}\, \rmL^{m-r}_k f_s \Big)_{s=0}
\quad\tx{and}\quad
  f_{\rmR,r}^{(d)}
\;:=\;
  \Big( \partial_s^d\, \frake_{r,m-r}\, \rmR^r_k f_s \Big)_{s=0}
\tx{,}
\end{gather*}
and let~$f_\rmL$ and~$f_\rmR$ be as in Propositions~\ref{prop:maass_form_change_weight_L} and~\ref{prop:maass_form_change_weight_R}.

Assume that~$k \le 0$ or~$k - m > 1$. Given an integer~$d \ge 0$ there are functions
\begin{gather*}
  f_\rmL^{(d)}
\in
  \lspan \CC\big\{
  f_{\rmL,r}^{(d-t)} \,:\, 0 \le r \le m, 0 \le t \le d
  \big\}
\end{gather*}
with
\begin{gather*}
  \Delta_{k-m}^d\, f_\rmL^{(d)} = f_\rmL
\tx{.}
\end{gather*}

Assume that~$k > 1$ or~$k + m < 1$. Given an integer~$d \ge 0$ there are functions
\begin{gather*}
  f_\rmR^{(d)}
\in
  \lspan \CC\big\{
  f_{\rmR,r}^{(d-t)} \,:\, 0 \le r \le m, 0 \le t \le d
  \big\}
\end{gather*}
with
\begin{gather*}
  \Delta_{k+m}^d\, f_\rmR^{(d)} = f_\rmR
\tx{.}
\end{gather*}
\end{theorem}

\begin{proof}
The case~$m = 0$ follows from Corollary~\ref{cor:spectral_derivative_laplace_classical}. We therefore assume that~$m$ is positive, and proceed with the following strategy. The cases of~$f^{(d)}_\rmL$ and~$f^{(d)}_\rmR$ differ only in the last step. We start by identifying the $\CC[\Delta_k]$-modules generated by~$f^{(d)}_{\rmL,r}$, $0 \le r \le m$, with quotients of~$\CC[T] \otimes V$ for a fixed complex vector space~$V$ of dimension~$m+1$, which is independent of~$d$. While this is merely a matter of renormalization, it allows us to relate the functions~$f^{(d)}_\rmL$ for varying~$d$ to each other. Specifically, it enables us to reformulate the statement of the theorem in terms of generalized eigenvectors. Using the grading with respect to powers of~$T$, we can then reduce our considerations to a problem concerning the interplay of three endomorphisms of~$V$. We solve it by inspection of ranks and images via a coordinate projection. The ideas for~$f^{(d)}_\rmR$ remain the same, but in the last step we use an alternating trace instead of a coordinate projection.

We consider the case of~$f_\rmL$. We let~$V = \CC^{m+1}$ with basis~$v_0, \ldots, v_m$ and set
\begin{gather*}
  W_d
:=
  \big( \CC[T] \slash T^{d+1} \big) \otimes V
\tx{.}
\end{gather*}
These are modules for~$R = \CC[T] \otimes \End(V)$ with~$R$-module homomorphisms
\begin{gather*}
  W_d \hra W_{d+1},\,
  w \mto T w
\quad\tx{and}\quad
  W_{d+1} \thra W_d,\,
  w \mto w \,\pmod{T^{d+1}}
\tx{.}
\end{gather*}
We identify~$V$ with~$W_0$.

The vector space isomorphisms
\begin{gather*}
  \phi_d :\,
  W_d
\lra
  \lspan \CC \big\{ f^{(d-t)}_{\rmL,r} \,:\, 0 \le r \le m, 0 \le t \le d \big\}
\tx{,}\;
  T^t v_r
\lmto
  \mfrac{1}{(d-t)!}\,
  f^{(d-t)}_{\rmL,r}
\end{gather*}
allow us to reformulate the formula for~$\Delta_{k-m}$ in Lemma~\ref{la:laplace_BKbasis_spectral_family} as an endomorphism of~$W_d$. Specifically, we have the pullback
\begin{gather*}
  \phi_d^\ast\, \Delta_{k-m}
=
  A + T B + T^2 C
\end{gather*}
with linear transformations~$A, B, C \in \End(V)$ defined by
\begin{align*}
  A v_r
&{}=
  (m-r)(m-2r-k)
  v_r
\,-\,
  v_{r-1}
\,+\,
  (r+1)(m-r) (m-r-1)(m-r-k)
  v_{r+1}
\tx{,}
\\
  B v_r
&{}=
  (1-k)
  v_r
\,+\,
  (r+1)(m-r) (1-k)
  v_{r+1}
\tx{,}
\\
  C v_r
&{}=
-\,
  v_r
\,-\,
  (r+1)(m-r)
  v_{r+1}
\tx{,}
\end{align*}
where we set~$v_{-1} = v_{m+1} = 0$ to simplify notation. Since~$d$ does not appear in these equations, we conclude that these pullbacks lift to an element~$\Delta$ of~$R$. We consider~$\Delta$ as an element of~$\End(W_d)$ for any~$d$, depending on the context, and observe it commutes with~$W_d \hra W_{d+1}$ and~$W_{d+1} \thra W_d$.

Since we assume that~$k \le 0$ or~$k > m$, the function~$f_\rmL$ defined in Proposition~\ref{prop:maass_form_change_weight_L} receives a contribution from~$\frake_{m,0} \rmL^0\, f$ and is hence nonzero. Translated to the current notation via~$\phi_0$, the same proposition the implies that the kernel of~$\Delta$ on~$W_0$, that is the kernel of~$A$, is spanned by by the preimage of~$f_\rmL$ under~$\phi_0$. Up to rescaling it equals a vector~$w_0 = \sum_r c_r v_r$, $c_r \in \CC$, with~$c_m = (1-k)$ and~$c_{m-1} = 1$. Note that when referring to~$c_{m-1}$, we use our assumption that~$m > 0$. We have~$k \ne 1$, since $k \le 0$ or~$k > m$.

From the formula for~$A v_r$ with~$r = m-1$ and~$r = m$, we infer that the image of~$A$ is spanned by~$v_0, \ldots, v_{m-1}$. Since~$w_0 \in \ker A$ receives a nontrivial contribution from~$v_m$, we conclude that~$\ker A = \ker A^2$. In other words, the generalized eigenspace of~$A$ associated with the eigenvalue zero has dimension one.

The action of~$\Delta$ on~$W_d$ preserves the filtration~$W_d \supset T W_{d-1} \supset \cdots T^d W_0$. From the expression~$\Delta = A + T B + T^2 C$, we see that its action on the associated graded space
\begin{gather*}
  W_d \slash T W_{d-1} \oplus T W_{d-1} \slash T^2 W_{d-2} \oplus \cdots \oplus T^d W_0
\end{gather*}
coincides with the action of~$A$. We conclude that zero is an eigenvalue of~$\Delta \in \End(W_d)$ with multiplicity~$d+1$. For clarity, we recall that the multiplicity of an eigenvalue is the exponent of the corresponding term in the factorization of the characteristic polynomial, while the exact depth is the exponent in the minimal polynomial minus one.

We next show the existence of generalized eigenvectors~$w_d \in W_d$ for~$\Delta \in \End(W_d)$ of eigenvalue~$0$ and exact depth~$d$ in~$W_d$. We record that this together with the previous multiplicity statement implies that the kernel of~$\Delta \in \End(W_d)$ is spanned by~$T^d w_0$. Note that the case~$d = 0$ follows with~$w_0 \in W_0$.

We consider the case~$d = 1$. Since~$w_0 \in W_0 \hra T W_0 \subset W_1$ lies in the kernel of~$\Delta$, it suffices to construct a preimage~$w_1$ of~$T w_0 \in W_1$ under~$\Delta$. To this end, we first verify by a calculation that~$((m+1-k) - B) w_0$ lies in the image of~$A$, which we recall is spanned by~$v_0, \ldots, v_{m-1}$. We have
\begin{gather*}
  B w_0
=
  B \big( \cdots + v_{m-1} + (1-k) v_m \big)
=
  \cdots + (m + 1 - k) (1 - k) v_m
\tx{.}
\end{gather*}
In particular, $((m+1-k) - B) w_0$ lies in the span of~$v_0, \ldots, v_{m-1}$ as desired. We let~$\wtd{w}_0$ be a preimage under~$A$, and set~$w_1 = (w_0 + T \wtd{w}_0) \slash (m+1-k) \in W_1$. The assumption~$k \le 0$ or~$k-m > 1$ guarantees that~$m+1-k \ne 0$, i.e.~$w_1$ is well-defined.

We have
\begin{gather*}
  \Delta \frac{w_0 + T \wtd{w}_0}{m+1-k}
\equiv
  \frac{(A + T B) (w_0 + T \wtd{w}_0}{m+1-k}
\equiv
  \frac{T (B w_0 + A \wtd{w}_0)}{m+1-k}
\equiv
  T w_0
  \;\pmod{T^2}
\tx{.}
\end{gather*}
That is, $w_1$ is indeed a preimage of~$T w_0$ and thus a generalized eigenvector of exact depth~$1$.

Now by induction on~$d \ge 2$, we assume that the desired vector~$w_{d-1}$ exists. We recall that~$\Delta \in \End(W_{d-1})$ has eigenvalue~$0$ with multiplicity~$d$. We record that this in conjunction with the depth of~$w_{d-1}$ implies that the kernel of~$\Delta^{d-1} \in \End(W_{d-1})$ is contained in~$T W_{d-2}$.

We will now show that every~$v \in W_d$ with~$\Delta^d v = 0$ lies in~$T W_{d-1} \subset W_d$. By contraposition, we suppose that there were~$v \in W_d \setminus T W_{d-1}$ with~$\Delta^d v = 0$. Then~$v \,\pmod{T}$ in~$W_d \slash T W_{d-1}$, i.e., the $0$\thdash{} coefficient of~$v$ with respect to~$T$, is nonzero, but lies in the kernel of~$\Delta^d$. Since~$\Delta$ acts on~$W_d \slash T W_{d-1} \cong W_0$ as~$A$ and since~$\ker A = \ker A^d$ is spanned by~$w_0$, the $0$\thdash{} coefficient of~$v$ is a nonzero multiple of~$w_0$. That is, after rescaling~$v$ we can assume that
\begin{gather*}
  v
=
  (m + 1 - k) w_0 + \sum_{i = 1}^d T^i v_i
\tx{.}
\end{gather*}

We next consider~$v \,\pmod{T^2}$ in~$W_d \slash T^2 W_{d-2} \cong W_1$. Recall that the generalized eigenspace of~$\Delta \in \End(W_1)$ associated with the eigenvalue~$0$ is $2$-dimensional and spanned by the vectors~$T w_0$ and~$w_1$. By matching coefficients of~$v_0$, we thus find that~$v \equiv w_1 + T c w_0 \,\pmod{T^2}$ for some~$c \in \CC$, and~$\Delta v \equiv  T w_0 \,\pmod{T^2}$.

Since~$\Delta^d v = 0$ by assumption, we see that~$\Delta v \in T W_{d-1}$ lies in the kernel of~$\Delta^{d-1}$. We recorded before that the kernel of~$\Delta^{d-1}$ acting on~$W_{d-1}$ is contained in~$T W_{d-2}$. Since further~$T W_{d-1} \cong W_{d-1}$ as~$\CC[\Delta]$-modules, we conclude that~$\Delta v \in T^2 W_{d-2}$. This implies that~$\Delta v \equiv 0 \,\pmod{T^2}$, a contradiction

We have shown by contradiction that the kernel~$\Delta^d$ acting on~$W_d$ is contained in~$T W_{d-1}$ and thus has dimension~$d$ by the multiplicity statement for the eigenvalue~$0$ of~$\Delta$ acting on the associated graded module. By the same multiplicity statement and the fact that~$A$ has one-dimensional kernel on~$W_d \slash T W_{d-1}$, we conclude that there is a generalized eigenvector~$w_d \in W_d$ for~$\Delta$ of exact depth~$d+1$. Its image under~$\phi_d$ yields~$f_\rmL^{(d)}$. This finishes the proof in the case of~$f_\rmL$.

In the case of~$f_\rmR$, we follow the same pattern with
\begin{gather*}
  \phi_d :\,
  W_d
\lra
  \lspan \CC \big\{ f^{(d-t)}_{\rmR,r} \,:\, 0 \le r \le m, 0 \le t \le d \big\}
\tx{,}\;
  T^t v_r
\lmto
  \mfrac{1}{(d-t)!}\,
  f^{(d-t)}_{\rmR,r}
\end{gather*}
and~$\Delta = A + T B + T^2 C \equiv \phi^\ast_d \Delta_{k+m}$, where
\begin{align*}
  A v_r
&{}=
-\,
  \big( r(m-2r-k) + m \big)
  v_r
\,+\,
  r (r-1+k)
  v_{r-1}
\,-\,
  (r+1) (m-r)
  v_{r+1}
\tx{,}
\\
  B v_r
&{}=
  (1-k)
  v_r
\,+\,
  (1-k)
  v_{r-1}
\tx{,}
\\
  C v_r
&{}=
-\,
  v_r
\,-\,
  v_{r-1}
\tx{.}
\end{align*}

The image of~$A$ consists of all vectors~$v = \sum c_r v_r$ with vanishing alternating trace
\begin{gather*}
  \tdtr(v)
:=
  \sum (-1)^r c_r
\tx{.}
\end{gather*}
Proposition~\ref{prop:maass_form_change_weight_R} provides an element $w_0$ in the kernel of~$A$. We want to verify that~$w_0$ does not lie in the image of~$A$, that is, $\tdtr(w_0) \ne 0$. In the case~$k \le -m$, we have~$k < 0$ and~$\tdtr(w_0) \ne 0$ follows from the inspection of the expression for~$f_\rmR$ in Proposition~\ref{prop:maass_form_change_weight_R}, since the coefficients of~$w_0$ with respect to~$v_0,\ldots,v_{m+1}$ have alternating sign. We consider the case~$k > 1$, which implies~$m > -k$. Then we need to check that the following expression does not vanish:
\begin{gather*}
  \tdtr(w_0)
=
  \sum_{r = 0}^m \mfrac{(-1)^r}{(m-r)! (r+k-1)!}
\tx{.}
\end{gather*}
We multiply by~$(-1)^{k-1} (m+k-1)!$ and shift~$r$ by~$k-1$ to obtain a tail of the binomial expansion of~$0 = (1-1)^{m+k-1}$:
\begin{gather*}
  \sum_{r = k-1}^{m+k-1} \mbinom{m+k-1}{r} (-1)^r
=
  - \sum_{r = 0}^{k-1} \mbinom{m+k-1}{r} (-1)^r
\tx{.}
\end{gather*}
If~$2m \le m+k-1$, the summands on the left hand side have monotone absolute value, otherwise the ones on the right hand side do. We conclude that the sum is nonzero as desired by grouping successive terms.

To finish the proof, we show the existence of generalized eigenvectors~$w_d \in W_d$ as before. We can choose
\begin{gather*}
  w_1
=
  \frac{\tdtr(w_0) (w_0 - T \wtd{w}_0)}{(1-k) c_0}
\tx{,}
\end{gather*}
where~$\wtd{w}_0$ is a preimage under~$A$ of~$((1-k)c_0 \slash \tdtr(w_0) - B) w_0$ and~$w_0 = \sum c_r v_r$. The inductive proof of the existence of~$w_d$ for~$d \ge 2$ follows the same line as before.
\end{proof}

\begin{remark}
\label{rm:spectral_families_altering_weight}
The maps~$W_{d-1} \hra W_d$ and~$W_d \thra W_{d-1}$ can be used to iteratively determine the vectors~$w_d$. More specifically, $w_0$ is given by Propositions~\ref{prop:maass_form_change_weight_L} and~\ref{prop:maass_form_change_weight_R}. Further, for~$d \ge 1$ we have~$w_d \equiv w_{d-1} \,\pmod{T^d}$ and~$\Delta w_d$ lies in the span of~$T^t w_{d-t}$ for~$1 \le t \le d$. That is, we can set~$w_d = w_{d-1} + T^d v$ for some~$v \in V$ that is unique up to scalar multiples of~$w_0$, and determine~$A v$ uniquely from
\begin{gather*}
  \Delta w_{d-1} + T^d A v
\in
  \img \Delta
  \cap
  \lspan \CC\{ T^t w_{d-t} \,:\, 1 \le t \le d \}
\tx{.}
\end{gather*}
Note that in this step we use~$\img A \cap \ker A = \{0\}$.
\end{remark}

\section{Modular realizations}%
\label{sec:modular_realization}

We next provide modular realizations for each representation that we found in Section~\ref{ssec:classification_representation_labels} by providing the correct input to Corollary~\ref{cor:spectral_derivative_laplace_classical} and Theorem~\ref{thm:spectral_families_altering_weight}. Both require a spectral family~$f_s$. We primarily use spectral families that specialize at~$s = 0$ to the modular realizations provided by Bringmann--Kudla in the case of~$d = 0$. Case~IIId, which does not occur for~$d = 0$, can be constructed from Case~Ib.

Many of the examples in this section are products of functions that are modular co-variant. Note that by the remark following Definition~\ref{def:polyharmonic_maass_forms} both the modular co-variance and the growth condition on polyharmonic weak Maa\ss{} forms are compatible with products. Our application of Theorem~\ref{thm:spectral_families_altering_weight} then ensures that the products that we encounter in this section are also polyharmonic.

We calculated the examples in this section via an implementation of Remark~\ref{rm:spectral_families_altering_weight} in the computer algebra system Nemo~\cite{fieker-hart-hofmann-johansson-2017}. The code is available on the third named author's homepage. We use the notation
\begin{gather}
\begin{alignedat}{2}
  E_{k,s_0}^{(d)}(\tau)
&=
  \big( \partial_s^d E_k(\tau,s) \big)_{s=s_0}
\tx{,}\qquad
&
  E_k^{(d)}(\tau)
&=
  E_{k,0}^{(d)}(\tau)
\tx{,}\\
  P_{k,n,s_0}^{(d)}(\tau)
&=
  \big( \partial_s^d F_{k,n}(\tau,1-\tfrac{k}{2}-s) \big)_{s=s_0}
\tx{,}\qquad
&
  P_{k,n}^{(d)}(\tau)
&=
  P_{k,n,0}^{(d)}(\tau)
\tx{.}
\end{alignedat}
\end{gather}
The primary purpose of our implementation is to determine linear combinations of these functions that fall under Cases~Ia--IIId.

The cases in this section are labeled in a compatible way with Bringmann--Kudla; See Table~\ref{tab:label_translation} for a translation between these labels and the representation theoretic labels that we employed in Section~\ref{ssec:classification_representation_labels}.

\subsection{Case Ia}%
\label{sssec:modular_realization:Ia}

We provide a polyharmonic weak Maa\ss{} form~$f$ of exact depth~$d$ and weight~$k < 1$ with~$\rmL_k\, \Delta_k^d\, f = 0$ and~$\rmR_k^{1-k}\, \Delta_k^d\, f = 0$.

The case of~$k = 0$ and~$d = 0$ is classical: We have the modular form~$f(\tau) = 1$, which vanishes under~$\rmL_0$ and~$\rmR_0$. Bringmann and Kudla extended this to all~$k \le 0$ via a vector-valued construction, which also falls under the scope of Proposition~\ref{prop:maass_form_change_weight_L}. Specifically, we have a modular form~$\frake_{-k,0}$ of weight~$k \le 0$ that vanishes under~$\rmL_k$ and~$\rmR_k^{1-k}$ by~\eqref{eq:BKbasis_maass_operators}.

Theorem~\ref{thm:spectral_families_altering_weight} gives the existence of preimages of~$\frake_{-k,0}$ under~$\Delta_k^d$ when setting~$f_s =E_0(\,\cdot\,,s)$, $m=-k$, and~$k = 0$.

\begin{example}
\label{ex:modular_realization_Ia}
We obtain a modular realization~$f^{(2)}$ of this case in weight~$-3$ and depth~$2$ using pure Nemo code with~$f^{(0)} = \Delta_{-3}^2\, f^{(2)}$ given by
\begin{align*}
&\hphantom{{}={}}
  \tfrac{1}{72}\,
  \frake_{0,3}\,
  \rmL^3 E_0^{(2)}
  \,+\,
  \tfrac{1}{8}\,
  \frake_{1,2}\,
  \rmL^2 E_0^{(2)}
  \,+\,
  \tfrac{1}{2}\,
  \frake_{2,1}\,
  \rmL^1 E_0^{(2)}
  \,+\,
  \tfrac{1}{2}\,
  \frake_{3,0}\,
  E_0^{(2)}
\\
&\qquad
  +\,
  \tfrac{11}{216}\,
  \frake_{0,3}\,
  \rmL^3 E_0^{(1)}
  \,+\,
  \tfrac{3}{8}\,
  \frake_{1,2}\,
  \rmL^2 E_0^{(1)}
  \,+\,
  \frake_{2,1}\,
  \rmL^1 E_0^{(1)}
\\
&{}=
 \tfrac{1}{2}\,
 \frake_{3,0}\,
 E^{(2)}_{0,0}
 \,+\,
 \tfrac{1}{18}
 \frake_{0,3}\,
 E^{(1)}_{-6,3}
 \,+\,
 \tfrac{1}{4}
 \frake_{1,2}\,
 E^{(1)}_{-4,2}
 \,+\,
 \frake_{2,1}\,
 E^{(1)}_{-2,1}
\\
&\qquad
 +\,
 \tfrac{5}{27}\,
 \frake_{0,3}\,
 E^{(0)}_{-6,3}
 \,+\,
 \tfrac{5}{8}
 \frake_{1,2}\,
 E^{(0)}_{-4,2}
 \,+\,
 \frake_{2,1}\,
 E^{(0)}_{-2,1}
\tx{.}
\end{align*}
\end{example}

\subsection{Case Ib}%
\label{sssec:modular_realization:Ib}

We provide a polyharmonic weak Maa\ss{} form~$f$ of exact depth~$d$ and weight~$k < 1$ with~$\rmL_k\, \Delta_k^d\, f = 0$, and~$\rmR_k^{1-k}\, \Delta_k^d\, f \ne 0$.

Bringmann--Kudla realized Case~Ib for~$d = 0$ in terms of weakly holomorphic modular forms (excluding constants in weight~$k = 0$). Indeed, a weakly holomorphic modular form~$f$ with nonzero principal part behaves as required under~$\rmL_k$ and~$\rmR_k^{1-k}$. We can obtain further  modular realizations~$\frake_{m,0} f$ in depth~$0$ and weight~$k - m$ by Proposition~\ref{prop:maass_form_change_weight_L} for such $f$ and a non-negative integer~$m$. Indeed, we have~$\rmL\, \frake_{m,0} f = \frake_{m,0}\, \rmL f = 0$ and~$\rmR^{1-k+m}\, \frake_{m,0} f$ can be written as a linear combination~$\sum_r c_r \frake_{r,m-r} \rmR^{1-k-r} f$, $c_r \in \CC$, with nonzero~$c_0$. Since we have~$\frake_{0,m}\, \rmR^{1-k} f \ne 0$ by the Bol Identity, we conclude that~$\rmR^{1-k+m}\, \frake_{m,0} f \ne 0$.

Given such a weakly holomorphic modular form~$f$ there is a holomorphic family~$f_s$, $s\in\CC$, with~$f = f_0$ and~$\Delta\, f_s = s (1 - s - k)\, f_s$ after substituting~$1 - \frac{k}{2} - s$ for~$s$ in Proposition~\ref{prop:poincarespectral}. This family can be explicitly constructed via Poincar\'e series. The desired modular realization of Case~Ib therefore exists for all~$k \le 0$ by Corollary~\ref{cor:spectral_derivative_laplace_classical} applied to this family~$f_s$ and given~$k$. The vector-valued generalizations~$\frake_{m,0} f_s$ give rise to preimages under~$\Delta^d$ of~$\frake_{m,0} f$ by Theorem~\ref{thm:spectral_families_altering_weight} applied to~$f_s$, $k$, and~$m$.

\subsection{Case Ic}%
\label{sssec:modular_realization:Ic}

We provide a polyharmonic weak Maa\ss{} form~$f$ of exact depth~$d$ and weight~$k < 1$ with~$\rmL_k\, \Delta_k^d\, f \ne 0$, and~$\rmR_k^{1-k}\, \Delta_k^d\, f = 0$.

Recall the flipping operator from~\eqref{eq:def:flipping_operator}. Given a polyharmonic Maa\ss{} form of exact depth~$d$ and weight~$k$ that realizes Case~Ib, $\rmF_k\, f$ realizes Case~Ic in the same depth and weight. The depth of~$f$ and~$\rmF_k\,f$ coincides by the commutator relation in~\eqref{eq:flipping_laplace_and_involution}. We apply Equations~\eqref{eq:flipping_maass_bol} and~\eqref{eq:flipping_maass_low} to find that
\begin{gather*}
  \rmL_k\, \Delta_k^d\, \rmF_k\, f
=
  \rmL_k\, \rmF_k\, \Delta_k^d\, f
=
  \mfrac{y^{2-k}}{(-k)!}\, \ov{\rmR^{1-k}\,\Delta_k^d f}
\ne
  0
\end{gather*}
and likewise~$\rmR_k^{1-k}\, \Delta_k^d\, \rmF_k\, f = 0$.

\subsection{Case Id}%
\label{sssec:modular_realization:Id}

We provide a polyharmonic weak Maa\ss{} form~$f$ of exact depth~$d$ and weight~$k < 1$ with~$\rmL_k\, \Delta_k^d\, f \ne 0$, and~$\rmR_k^{1-k}\, \Delta_k^d\, f \ne 0$.

As in the work of Bringmann--Kudla, we can obtain this form from polyharmonic weak Maa\ss{} forms~$f_{\mathrm{Ib}}$ and~$f_{\mathrm{Ic}}$ that realize Cases~Ib and~Ic. Then~$f = f_{\mathrm{Ib}} + f_{\mathrm{Ic}}$ is harmonic of depth~$d$ and satisfies
\begin{alignat*}{3}
  \rmL_k \Delta_k^d\, f
&{}=
  \rmL_k \Delta_k^d\, f_{\mathrm{Ib}} + \rmL_k \Delta_k^d\, f_{\mathrm{Ic}}
&&{}=
  \rmL_k \Delta_k^d\, f_{\mathrm{Ic}}
&&{}\ne
  0
\tx{,}
\\
  \rmR_k \Delta_k^d\, f
&{}=
  \rmR_k \Delta_k^d\, f_{\mathrm{Ib}} + \rmR_k \Delta_k^d\, f_{\mathrm{Ic}}
&&{}=
  \rmR_k \Delta_k^d\, f_{\mathrm{Ib}}
&&{}\ne
  0
\tx{.}
\end{alignat*}

Another construction is given by Eisenstein series. Recall that~$E_k(\,\cdot\,,1-k)$ is harmonic. We have~$\rmL\, E_k(\,\cdot\,,1-k) = E_{k-2}(\,\cdot\,,-k) \ne 0$ and~$\rmR^{1-k}_k\, E_k(\,\cdot\,,1-k) = (1-k)! E_{2-k} \ne 0$ by \eqref{eq:eis:maass_operators}.

For any non-negative integer~$m$, Proposition~\ref{prop:maass_form_change_weight_L} applied to~$E_k(\,\cdot\,,1-k)$ yields further realizations in weight~$k - m$. If~$m \le -k$, then Proposition~\ref{prop:maass_form_change_weight_R} yields yet further realizations in weight~$k + m$. Note that this includes the choice~$m = -k$, which provides modular realizations in weight~$0$.

If~$k < 0$ we obtain modular realizations in weight~$k$ and depth~$d$ from Corollary~\ref{cor:spectral_derivative_laplace_classical} applied to~$E_k(\,\cdot\,,1-k-s)$. The vector-valued realizations in depth~$0$ and weight~$k \pm m$ arising from~$E_k(\,\cdot\,,1-k)$ via Propositions~\ref{prop:maass_form_change_weight_L} and~\ref{prop:maass_form_change_weight_R}, can be extended to higher depth by virtue of Theorem~\ref{thm:spectral_families_altering_weight}.

\begin{remark}
The exceptional role of~$k = 0$ in Case~Id is connected to the realization of Case~Ia via~$E_0(\,\cdot\,,0)$.
\end{remark}

\subsection{Case IIa}%
\label{sssec:modular_realization:IIa}

We provide a polyharmonic weak Maa\ss{} form~$f$ of exact depth~$d$ and weight~$k = 1$ with~$\rmL_k\, \Delta_k^d\, f = 0$.

In analogy with Case~Ib, weakly holomorphic modular forms provide modular realizations of Case~IIa in depth~$d = 0$. Corollary~\ref{cor:spectral_derivative_laplace_classical} in conjunction with Proposition~\ref{prop:poincarespectral} allows us to extend them to positive depth. More specifically, given a weakly holomorphic modular form~$f$ of weight~$k = 1$, as in Case~Ib Proposition~\ref{prop:poincarespectral} yields a spectral family~$f_s$ with~$f_0 = f$ and~$\Delta_1\,f_s = s(1-1-s) f_s = -s^2 f_s$. Note that as opposed to Cases~Ia--Id, only even iterated derivatives~$\big(\partial_s^{2d-2t}\, f_s\big)_{s=0}$, $0 \le t \le d$, of~$f_s$ occur when applying Corollary~\ref{cor:spectral_derivative_laplace_classical} to~$f_s$.

\subsection{Case IIb}%
\label{sssec:modular_realization:IIb}

We provide a polyharmonic weak Maa\ss{} form~$f$ of exact depth~$d$ and weight~$k = 1$ with~$\rmL_k\, \Delta_k^d\, f \ne 0$.

Bringmann--Kudla realized depth~$0$ of Case~IIb via incoherent Eisenstein series associated to prime fundamental discriminants~$-D < 0$ and the function
\begin{gather*}
  \phi_D^-\big(\begin{psmatrix} a & b \\ c & d \end{psmatrix}\big)
=
  -i \sqrt{D} \big(\mfrac{-D}{c}\big)
  \tx{, if\ } \gcd(D, c) = 1
\tx{;}\quad
  \phi_D^-\big(\begin{psmatrix} a & b \\ c & d \end{psmatrix}\big)
=
  \big(\mfrac{-D}{a}\big)
  \tx{, otherwise,}
\end{gather*}
where~$(-D \slash c)$ is the quadratic symbol. In the next definition, the normalization of~$s$ differs from the one in~\cite{BK} in order to obtain a spectral family that satisfies the usual eigenvalue equation~$\Delta_1 E^-_D(\tau,s) = s(1-1-s) E^-_D(\tau,s) = -s^2 E^-_D(\tau,s)$:
\begin{gather*}
  E^-_D(\tau,s)
:=
  \sum_{\ga \in \Ga_\infty \backslash \SL{2}(\ZZ)}
  \phi_D^-(\ga)\,
  y^s \big|_1 \ga
\tx{.}
\end{gather*}
Since~$E^-_D(\tau,s)$ vanishes at~$s = 0$, the next definition features the exponent~$d+1$ as opposed to~$d$. We set
\begin{gather*}
  E^{-(d)}_D(\tau)
:=
  \big( \partial_s^{d+1}\, E^-_D(\tau) \big)_{s=0}
\tx{.}
\end{gather*}

A modular realization for~$d = 0$ is given by~$E^{-(0)}_D(\tau)$. We can apply Corollary~\ref{cor:spectral_derivative_laplace_classical} to the spectral family~$E^-_D(\,\cdot\,,s) \slash s$ in weight~$k = 1$ to obtain modular realizations in positive depth~$d$ from linear combinations of~$E^{-(2d-2t)}_D$ for~$0 \le t \le d$.

\subsection{Case IIIa}%
\label{sssec:modular_realization:IIIa}

We provide a polyharmonic weak Maa\ss{} form~$f$ of exact depth~$d$ and weight~$k > 1$ with~$\rmL_k^k \Delta_k^{d-1}\, f \ne 0$, if~$d > 0$, and~$\rmL_k\, \Delta_k^d\, f = 0$.

A modular realization in depth~$0$ is provided by holomorphic modular forms. For general~$d$, let~$f$ be a modular realization of Case~Id in depth~$d > 0$ and weight~$2-k$. We claim that then~$\rmR^{k-1}_{2-k} \,f$ is a modular realization of Case~IIIa in depth~$d$ and weight~$k$. Indeed, by~\eqref{eq:laplace_commutators} and~\eqref{eq:maass_iterated_commutators} we have
\begin{alignat*}{2}
  \rmL_k\, \Delta_k^d\, \rmR_{2-k}^{k-1}\, f
&{}=
  \rmL_k\, \rmR_{2-k}^{k-1}\, \Delta_{2-k}^d\, f
&&{}=
  - \rmR_{2-k}^{k-2}\, \Delta_{2-k}^{d+1}\, f
=
  0
\tx{,}
\\
  \rmL_k^k\, \Delta_k^{d-1}\, \rmR_{2-k}^{k-1}\, f
&{}=
  \rmL_k^k\, \rmR_{2-k}^{k-1}\, \Delta_{2-k}^{d-1}\, f
&&{}=
  (k-1)! (k-2)!\,
  \rmL_{2-k}\,
  \Delta_{2-k}^d\, f
\ne
  0
\tx{.}
\end{alignat*}

\begin{remark}
A spectral deformation of holomorphic cusp forms can be achieved directly via Theorem~\ref{thm:spectral_families_altering_weight} applied to the Poincar\'e series in Theorem~3.1 and~3.4 of~\cite{fay-1977}.
\end{remark}

\subsection{Case IIIb}%
\label{sssec:modular_realization:IIIb}

We provide a polyharmonic weak Maa\ss{} form~$f$ of exact depth~$d$ and weight~$k > 1$ with~$\rmL_k \Delta^d\, f \ne 0$ and~$\rmL_k^k\, \Delta_k^d\, f = 0$.

A modular realization in depth~$0$ and weight~$2$ is given by the weight-$2$ Eisenstein series~$E_2$. We have~$\rmL\, E_2 = \frac{3}{\pi} \ne 0$ and~$\rmL^2\, E_2 = 0$ as required. For any positive integer~$m$ Proposition~\ref{prop:maass_form_change_weight_R} applied to~$E_2$ yields a modular realization in weight~$2 + m$, which has already appeared in the work of Bringmann--Kudla:
\begin{gather*}
  \sum_{r=0}^m
  \mfrac{1}{(m-r)! (r+1)!}\,
  \frake_{r,m-r} \rmR_2^r E_2
\tx{.}
\end{gather*}

We employ Theorem~\ref{thm:spectral_families_altering_weight} to the spectral family~$E_2(\,\cdot\,,s)$ of weight~$k = 2$ and any non-negative integer~$m$ to obtain modular realization of given depth~$d$ and weight~$2 + m$.

\subsection{Case IIIc}%
\label{sssec:modular_realization:IIIc}

We provide a polyharmonic weak Maa\ss{} form~$f$ of exact depth~$d$ and weight~$k > 1$ with~$\rmL_k^k\, \Delta_k^d\, f \ne 0$.

Bringmann--Kudla utilized Poincar\'e series to provide a modular realization. They fit, more generally, in the following framework. Let~$f$ be a modular realization of Case~Ic in depth~$d+1$ and weight~$2-k$. We claim that then~$\rmR^{k-1}_{2-k}\,f$ is a modular realization of Case~IIIc in depth~$d$ and weight~$k$. Similar to the treatment of Case~IIIa, we employ~\eqref{eq:laplace_commutators} and~\eqref{eq:maass_iterated_commutators} to find that if~$d > 0$ then
\begin{align*}
  \rmL_k^k\, \Delta_k^d\, \rmR_{2-k}^{k-1}\, f
&{}=
  \rmL_k^k\, \rmR_{2-k}^{k-1}\, \Delta_{2-k}^d\, f
=
  (k-1)! (k-2)!\,
  \rmL_k\,
  \Delta_{2-k}^{d+1}\, f
\ne
  0
\tx{,}
\\
  \Delta_k^{d+1}\, \rmR_{2-k}^{k-1}\, f
&{}=
  \rmR_{2-k}^{k-1}\, \Delta_{2-k}^{d+1}\, f
=
  0
\tx{.}
\end{align*}

\subsection{Case IIId}%
\label{sssec:modular_realization:IIId}

We provide a polyharmonic weak Maa\ss{} form~$f$ of exact depth~$d > 0$ and weight~$k > 1$ with~$\rmL^k\, \Delta^{d-1}\, f = 0$. Recall that this case does not appear in depth~$0$ and thus is absent from the classification of Bringmann--Kudla.

Let~$f$ be a modular realization of Case~Ib in depth~$d+1$ and weight~$2-k$. We claim that then~$\rmR^{k-1}_{2-k}\,f$ is a modular realization of Case~IIId in depth~$d$ and weight~$k$. Repeating the computation of Case~IIIc based on \eqref{eq:laplace_commutators} and~\eqref{eq:maass_iterated_commutators}, we confirm that
\begin{align*}
  \rmL_k^k\, \Delta_k^{d-1}\, \rmR_{2-k}^{k-1}\, f
&{}=
  \rmL_k^k\, \rmR_{2-k}^{k-1}\, \Delta_{2-k}^{d-1}\, f
=
  (k-1)! (k-2)!\,
  \rmL_{2-k}\,
  \Delta_{2-k}^d\, f
=
  0
\tx{,}
\\
  \Delta_k^d\, \rmR_{2-k}^{k-1}\, f
&{}=
  \rmR_{2-k}^{k-1}\, \Delta_{2-k}^d\, f
\ne
  0
\tx{.}
\end{align*}
Observe that the first equality also implies that~$\rmR_{2-k}^{k-1}\,f$ vanishes under~$\Delta_k^{d+1}$.

\smallskip
\noindent
\emph{Acknowledgement}. We thank the anonymous referee for a thorough read and his/her helpful suggestions and remarks.


\ifbool{nobiblatex}{
  \sloppy
  \bibliographystyle{alpha}
  \bibliography{bibliography.bib}
}{%
  \sloppy
  \Needspace*{4em}
  \printbibliography[heading=none]
}


\Needspace*{3\baselineskip}
\noindent
\rule{\textwidth}{0.15em}

{\noindent\small
Universität Bielefeld, Fakultät für Mathematik, Postfach 100 131, 33501 Bielefeld, Germany\\
E-Mail: \url{alfes@math.uni-bielefeld.de}
}\vspace{.5\baselineskip}

{\noindent\small
Universit\"at Paderborn,
Mathematisches Institut,
Warburger Stra\ss{}e 100,
33098 Paderborn,
Germany\\
E-mail: \url{burban@math.uni-paderborn.de}
}\vspace{.5\baselineskip}

{\noindent\small
Chalmers tekniska högskola och G\"oteborgs Universitet,
Institutionen f\"or Matematiska vetenskaper,
SE-412 96 G\"oteborg, Sweden\\
E-mail: \url{martin@raum-brothers.eu}\\
Homepage: \url{https://martin.raum-brothers.eu}
}


\ifdraft{%
\listoftodos%
}

\end{document}
